\documentclass[11pt]{amsart}

\usepackage[a4paper,hmargin=3.5cm,vmargin=4cm]{geometry}
\usepackage{amsfonts,amssymb,amscd,amstext}
\usepackage{graphicx}
\usepackage[dvips]{epsfig}

\usepackage{fancyhdr}
\usepackage{times}
\usepackage{enumerate}
\usepackage{titlesec}
\usepackage{mathrsfs}

\input xy
\xyoption{all}

\titleformat{\section}
{\filcenter\bfseries\large} {\thesection{.}}{0.2cm}{} 
\titleformat{\subsection}[runin]
{\bfseries} {\thesubsection{.}}{0.15cm}{}[.]
\titleformat{\subsubsection}[runin]
{\em}{\thesubsubsection{.}}{0.15cm}{}[.]

\usepackage[up,bf]{caption}

\newtheorem{theorem}{Theorem}[section]
\newtheorem{proposition}[theorem]{Proposition}

\newtheorem{lemma}[theorem]{Lemma}
\newtheorem{corollary}[theorem]{Corollary}

\theoremstyle{definition}
\newtheorem{definition}[theorem]{Definition}
\newtheorem{remark}[theorem]{Remark}
\newtheorem{question}[theorem]{Question}

\newtheorem{example}[theorem]{Example}

\numberwithin{equation}{section}
\numberwithin{figure}{section}

\def\Fcal{\mathcal{F}}
\def\Gcal{\mathcal{G}}

\def\Lcal{\mathcal{L}}
\def\Pcal{\mathcal{P}}
\def\Qcal{\mathcal{Q}}

\def\Tcal{\mathcal{T}}

\def\be{\mathbf{e}}
\newcommand\E{\mathrm{e}}

\def\Cscr{\mathscr{C}}

\def\Oscr{\mathscr{O}}

\def\c{\mathbb{C}}

\def\d{\mathbb{D}}
\def\cd{\overline{\mathbb{D}}}

\def\r{\mathbb{R}}
\def\n{\mathbb{N}}
\def\t{\mathbb{T}}

\def\g{\mathbb{G}}

\def\Agot{\mathfrak{A}}

\def\Ngot{\mathfrak{N}}
\def\Mgot{\mathfrak{M}}

\def\igot{\mathfrak{i}}

\def\M{\mathbf{M}}

\def\dist{\mathrm{dist}}
\def\span{\mathrm{span}}
\def\reg{\mathrm{reg}}
\def\sing{\mathrm{sing}}

\def\supp{\mathrm{supp}}
\def\Psh{\mathrm{Psh}}
\def\NPsh{\mathrm{\mathfrak{N}Psh}}
\def\MPsh{\mathrm{\mathfrak{M}Psh}}
\def\Hess{\mathrm{Hess}\,}
\def\tr{\mathrm{tr}}

\def\Co{\mathrm{Co}}

\newcommand\wh{\widehat}

\newcommand\di{\partial}
\newcommand\dibar{\overline\partial}

\begin{document}

\fancyhead[LO]{Minimal hulls of compact sets in $\r^3$}
\fancyhead[RE]{B.\ Drinovec Drnov\v sek \& F.\ Forstneri\v c}
\fancyhead[RO,LE]{\thepage}

\thispagestyle{empty}


\vspace*{1cm}
\begin{center}
{\bf\LARGE Minimal hulls of compact sets in $\r^3$}

\vspace*{0.5cm}

{\large\bf Barbara Drinovec Drnov\v sek \& Franc Forstneri\v c }
\end{center}

\vspace*{1cm}

\begin{quote}
{\small
\noindent {\bf Abstract}\hspace*{0.1cm}   
The main result of this paper is a characterization of the minimal surface hull 
of a compact set $K$ in $\r^3$ by  sequences of conformal minimal discs
whose boundaries converge to $K$ in the measure theoretic sense,
and also by $2$-dimensional  minimal currents which are limits of Green currents supported
by conformal minimal discs. Analogous results are obtained 
for the null hull of a compact subset of $\c^3$. 
We also prove a null hull analogue of the  Alexander-Stolzenberg-Wermer theorem 
on  polynomial hulls of compact sets of finite linear measure, 
and a polynomial hull version of classical Bochner's tube theorem.

\vspace*{0.1cm}

\noindent{\bf Keywords}\hspace*{0.1cm} minimal surfaces, minimal hulls, holomorphic null  curves, 
null hulls, plurisubharmonic functions,  polynomial hulls, positive currents.

\vspace*{0.1cm}

}
\end{quote}



\section{The role of hulls in complex analysis and minimal surface theory}
\label{sec:intro}

When discussing hulls in various geometries, one typically deals with dual sets of objects.
Given a set  $\Pcal$ of real functions on a manifold $X$, the $\Pcal$-hull of a compact subset 
$K\subset X$ is 
\begin{equation}\label{eq:Phull}
	\widehat K_\Pcal= \{x\in X: f(x) \le \sup_K f \quad\forall f\in \Pcal\}.
\end{equation}
Suppose that $\Gcal$ is a class of geometric objects in $X$ (for example, submanifolds or subvarieties) 
such that the restriction $f|_C$ satisfies the maximum principle for every $f\in \Pcal$ and $C\in \Gcal$.
Then  $C\subset \wh K_\Pcal$ for every $C\in \Gcal$ with boundary $bC \subset K$, and the main question  is how closely is the hull  described  by such objects. 
A basic example is the convex hull $\Co(K)$ of a compact set in 
an affine space $X\cong \r^n$; here $\Pcal$ is the class of all affine linear functions 
on $X$ and $\Gcal$ is the collection of straight line segments in $X$. 

Our principal aim is to introduce and study a suitable notion of the
{\em minimal hull}, $\wh K_\Mgot$, of a compact set $K$ in $\r^3$. 
The idea is that $\widehat K_\Mgot$ should contain every compact
$2$-dimensional minimal surface $M\subset \r^3$ 
with boundary $bM$ contained in $K$ and hopefully not much more. Any such minimal surface is a solution of the 
{\em Plateau problem with free boundary} in $K$; for a closed Jordan curve $K$ 
we have the classical Plateau problem (see e.g.\ \cite{Fomenko,Dierkes2010}).
We define $\wh K_\Mgot$ by using the class of {\em minimal plurisubharmonic functions}
(Definition \ref{def:minimalpsh}). Minimal hull coincides with the {\em $2$-convex hull} of 
Harvey and Lawson \cite[Definition 3.1, p.\ 157]{Harvey-Lawson2013}.
We obtain three characterizations of the minimal hull: 
by sequences of conformal minimal discs (Corollary \ref{cor:minullhull}), 
by minimal Jensen measures (Corollary \ref{cor:minimalJensen}), and 
by Green currents (Theorem \ref{th:minimalcurrents} and Corollaries \ref{cor:currentnullhull}
and \ref{cor:hessian}).  The only reason for restricting to $\r^3$ is that
the main technical tool (see Lemma \ref{prop:RHnull})  is currently only available 
in dimension $3$. 

Harvey and Lawson studied minimal current hull in an arbitrary Riemannian manifold $(X,g)$
and showed that it is contained in the hull defined by 
$p$-plurisubharmonic functions for the appropriate value of $p$ with 
$2\le p<\dim X$  \cite[Sec.\ 4]{Harvey-Lawson2013} (see Remark 
\ref{rem:Mhull} below).  Our results show that these hulls coincide in $\r^3$,
but it is not clear whether they coincide in more general manifolds.

By way of motivation we recall the classical case  of the  {\em polynomial hull}, 
$\widehat K$, of a compact set $K$ in a complex Euclidean space $\c^n$.
This is the hull (\ref{eq:Phull}) with respect to the family $\Pcal=\{|f|: f\in \Oscr(\c^n)\}$, 
where $\Oscr(\c^n)$ is the algebra of all holomorphic functions on $\c^n$. 
The same hull is obtained by using the bigger class $\Psh(\c^n)$ of all plurisubharmonic functions 
(see Stout \cite[Theorem 1.3.11, p.\ 27]{Stout-convexity}). 
A natural dual class $\Gcal$ consists of complex curves.
The question to what extent is $\widehat K$ described by bounded complex curves 
in $\c^n$ with boundaries in $K$ has been an important driving force in the 
development of complex analysis.   Wermer  \cite{Wermer1958} and
Stolzenberg \cite{Stolzenberg1966} proved that, 
if $K$ is a union of finitely many compact smooth curves in $\c^n$, 
then $A=\widehat K\setminus K$ is a (possibly empty) one dimensional 
closed complex subvariety of $\c^n\setminus K$. 
Alexander extended this to compact sets of finite linear measure  \cite{Alexander1971}. 
Positive results are also known for certain embedded $2$-spheres in $\c^2$ \cite{Bedford-K}
and totally real tori in $\c^2$ \cite{Forstneric1988},
among others; a survey can be found in  \cite{Stout-convexity}.  
However, Stolzenberg's example \cite{Stolzenberg1963} shows that one must 
in general relax the requirement that the boundaries of curves lie {\em exactly} in $K$. 
The right notion was found by Poletsky who characterized the polynomial hull by bounded sequences 
of holomorphic discs with boundaries converging  to $K$ in 
measure theoretic sense; here is the precise result.

%
%
%
%
\begin{theorem}[\cite{Poletsky1991,Poletsky1993}]  \label{th:Poletsky}
Let $K$ be a compact set in $\c^n$, and let $B\subset \c^n$ be a ball containing $K$. 
A point $p\in B$ belongs to the polynomial hull 
$\widehat K$ if and only if there exists a sequence of holomorphic discs 
$f_j\colon \overline \d \to B$ satisfying the following for every $j=1,2,\ldots$:
\begin{equation}\label{eq:discs}
	f_j(0)=p \ \ \text{and} \ \ \big| \{t\in [0,2\pi]: \dist(f(e^{\imath t}),K)<1/j\}  \big| \ge 2\pi -1/j.
\end{equation}
Here $|I|$ denotes the Lebesgue measure of a set $I\subset \r$.
\end{theorem}

The fact that Poletsky's theorem also gives a simple proof of the following result
of Duval and Sibony \cite{Duval-Sibony1995} was explained by Wold in \cite{Wold2011}.

%
%
%
%
\begin{theorem}[\cite{Duval-Sibony1995,Wold2011}] \label{th:DS} 
Let $K$ be a compact set in $\c^n$. A point $p\in \c^n$ belongs to  the polynomial 
hull $\widehat K$ of $K$ if and only
if there exists a positive current $T$ on $\c^n$ of bidimension $(1,1)$
with compact support  such that $dd^c T = \mu - \delta_p$, 
where $\mu$ is a representative Jensen measure for (evaluation at) $p$.
\end{theorem}

The characterization  in Theorem \ref{th:DS}  was extended by Harvey and Lawson \cite{Harvey-Lawson20091}  
to hulls in calibrated geometries. Results in this paper are not of this type.

We obtain analogous  characterizations of minimal hulls of compact sets in $\r^3$ 
and null hulls of compact sets in $\c^3$.  A suitable class  of functions to define the minimal hull 
is  the following (see Definition \ref{def:minimalpsh}):

{\em An upper semicontinuous function $u\colon \omega\to\r\cup\{-\infty\}$ 
on a domain $\omega\subset\r^n$ is said to be
{\em minimal plurisubharmonic} if the restriction of $u$ to any affine $2$-dimensional 
plane $L\subset \r^n$ is subharmonic on $L\cap \omega$ 
(in any isothermal coordinates on $L$).}

The set of all such functions is denoted by $\MPsh(\omega)$.
For every $u\in \MPsh(\omega)$ and every conformal minimal
disc $f\colon \d\to\omega$ the composition $u\circ f$ is a subharmonic 
function on $\d$ (cf.\ Lemma \ref{lem:characterizationMS}), so
minimal surfaces form a class of objects which is dual to the class 
of minimal plurisubharmonic functions.  A $\Cscr^2$ function $u$ 
is minimal plurisubharmonic if and only if the sum of the two smallest eigenvalues of its Hessian
is nonnegative at every point; hence $\Cscr^2$ minimal plurisubharmonic functions are exactly 
{\em $2$-plurisubharmonic functions} studied by Harvey and Lawson 
\cite[Definition 2.2, p.\ 153]{Harvey-Lawson2013}.
We adopt a more suggestive terminology to emphasize their relationship with minimal surfaces.
We define the minimal hull, $\widehat K_\Mgot$, of a compact set $K\subset \r^n$
as the hull  (\ref{eq:Phull}) with respect to the family $\Pcal = \MPsh(\r^n)$; see Definition \ref{def:minimalhull}.
This notion coincides with the {\em $2$-convex hull} of Harvey and Lawson 
\cite[Definition 3.1, p.\ 157]{Harvey-Lawson2013}.
In analogy with Theorem \ref{th:Poletsky} we characterize  
the minimal hull of a compact set $K\subset \r^3$ by sequences of conformal minimal discs whose
boundaries converge to $K$ in the measure theoretic sense.

%
%
%
%
\begin{theorem}[Corollary \ref{cor:minullhull}]\label{th:minullhull}
Let $K$ be a compact set in $\r^3$, and let  $\omega\Subset \r^3$ be a bounded open convex set 
containing $K$. A point $p\in \omega$ belongs to the minimal hull $\widehat K_\Mgot$ of $K$
if and only if there exists a sequence of conformal minimal discs 
$f_j\colon \overline \d\to \omega$ such that for all $j=1,2,\ldots$ we have $f_j(0)=p$ and
\[
	\big| \{t\in [0,2\pi]: \dist(f_j(e^{\imath t}),K)<1/j\}  \big| \ge 2\pi -1/j.
\] 
\end{theorem}

Theorem \ref{th:minullhull} is also used to characterize the minimal hull by 
limits of Green currents; see Theorem \ref{th:minimalcurrents} and Corollaries \ref{cor:currentnullhull} 
and \ref{cor:hessian}.

In the proof of Theorem \ref{th:minullhull}  we use the following connection 
between conformal minimal discs in $\r^n$ and {\em holomorphic null discs}
in $\c^n$; see Osserman \cite{Osserman}. A smooth conformal immersion $g\colon \d\to\r^n$
is minimal if and only if it is harmonic, $\triangle g=0$. The map $g$ admits 
a harmonic conjugate $h$ on $\d$ such that
$f=g+\imath h= (f_1,\ldots,f_n)\colon \d\to\c^n$ is a holomorphic immersion
satisfying the identity
\[
	f'_1(\zeta)^2 + f'_2(\zeta)^2 +\cdots + f'_n(\zeta)^2 = 0\qquad \forall \zeta\in \d.
\]
Such $f$ is said to be a {\rm holomorphic null disc} in $\c^n$. 
Conversely, the real and the imaginary part  of a holomorphic null disc in $\c^n$ are conformal 
minimal discs in $\r^n$. There is a corresponding relationship between  minimal plurisubharmonic functions 
on a domain $\omega\subset\r^n$ and {\em null plurisubharmonic functions} 
on the tube $\Tcal_\omega=\omega\times \imath\r^n\subset  \c^n$
(see Definition \ref{def:npsh} and Lemma \ref{lem:minimal}).
For any null plurisubharmonic function $u$ on a domain $\Omega\subset\c^n$ 
and null holomorphic disc $f \colon \d\to\Omega$
the composition $u\circ f$ is a subharmonic function on $\d$
(cf.\ Proposition \ref{lem:smoothing}-(iii)). The {\em null hull}, $\widehat K_\Ngot$, 
of a compact set $K\subset\c^n$, is the hull (\ref{eq:Phull}) with respect to the family $\Pcal=\NPsh(\c^n)$; 
see  Definition \ref{def:nullhull}. We obtain the following characterization of
the null hull of a compact set in $\c^3$ by sequences of null holomorphic discs,
in analogy with Theorems \ref{th:Poletsky} and \ref{th:minullhull}.

%
%
%
%
\begin{theorem}[Corollary \ref{cor:DD}] \label{th:DD}
Let $K$ be a compact set in $\c^3$ and let  $\Omega\subset \c^3$ be a
bounded pseudoconvex Runge domain containing $K$.
A point $p\in \Omega$ belongs to the null hull $\widehat K_\Ngot$ of $K$ 
if and only if there exists a sequence of null holomorphic discs 
$f_j\colon \overline \d\to \Omega$ such that for all $j=1,2,\ldots$ we have 
$f_j(0)=p$ and
\begin{equation*} 
	\big| \{t\in [0,2\pi]: \dist(f_j(e^{\imath t}),K)<1/j\}  \big| \ge 2\pi -1/j.
\end{equation*}
\end{theorem}

By using this result we also characterize the null hull by limits of Green currents 
supported on null holomorphic discs; see Theorem \ref{th:nullcurrents}.
An important ingredient is Lemma \ref{lem:mass} which shows that 
the mass of the Green current on $\r^n$ supported by a conformal minimal disc
$f\colon \cd\to\r^n$ is bounded by the $L^2$-norm of $f$ on the circle $\t=b\d$.

We conclude this introduction by mentioning another line of results obtained in the 
paper.  Assume that $\omega$ is a domain in one of the spaces $\c^n$, $\c^3$, or $\r^3$. 
Let $\Pcal(\omega)$ denote one of the sets $\Psh(\omega)$, $\NPsh(\omega)$, or $\MPsh(\omega)$.
Let $\Gcal(\omega)$ denote the dual set of discs $\cd\to\omega$: 
holomorphic if $\omega\subset\c^n$ and $\Pcal(\omega)=\Psh(\omega)$,
null holomorphic if $\omega\subset\c^3$ and $\Pcal(\omega)=\NPsh(\omega)$, 
and conformal minimal if $\omega\subset\r^3$ and $\Pcal(\omega)=\MPsh(\omega)$. 
The following result is the main ingredient in the proof of all the theorems 
mentioned so far; it furnishes many nontrivial examples of functions 
in classes under consideration.

%
%
%
%
\begin{theorem}  \label{th:Pminorant} 
Assume that $\omega$, $\Pcal(\omega)$ and $\Gcal(\omega)$ are as above. 
Let $\phi \colon \omega\to \r\cup\{-\infty\}$ be an upper semicontinuous  function. 
Then the function $u:\omega\to \r\cup\{-\infty\}$, given by
\begin{equation}
\label{eq:PF}
   u(x) = \inf \Big\{\int^{2\pi}_0 \phi(f(\E^{\imath t}))\, 
   \frac{dt}{2\pi} \colon \ f\in \Gcal(\omega),\ f(0)=x \Big\}, \quad x\in \omega,
\end{equation}
belongs to $\Pcal(\omega)$  or is identically $-\infty$; 
moreover, $u$ is the supremum of functions in $\Pcal(\omega)$  which are 
bounded above by $\phi$. 
\end{theorem}

The basic case of Theorem  \ref{th:Pminorant}, with $\omega$ a domain 
in $\c^n$ and $\Pcal(\omega)=\Psh(\omega)$, is a fundamental result 
due to Poletsky \cite{Poletsky1991,Poletsky1993} and Bu and Schachermayer \cite{Bu-Schachermayer}.
(For generalizations see L\'arusson and Sigurdsson \cite{Lar-Sig1998,Lar-Sig2003},
Rosay \cite{Rosay1,Rosay2}, Drinovec Drnov\v sek and Forstneri\v c \cite[Theorem 1.1]{DF2012},
and Kuzman \cite{Kuzman}, among others.) 
The other cases are new and proved in this paper; see Theorem \ref{th:nullpshminorant} 
for null plurisubharmonic functions, and Theorem \ref{th:minpshminorant} 
for minimal plurisubharmonic functions. 

One of the main ingredients in the proof  of Theorem \ref{th:Pminorant}  is 
the approximate solution of the Riemann-Hilbert boundary value problem 
for discs in the respective classes $\Gcal(\omega)$. This result  is well known
for holomorphic discs, but  it has been established only very recently 
for null holomorphic discs (Alarc\'on and Forstneri\v c \cite{AF2013}).
The case of conformal minimal discs can be reduced to null holomorphic discs.

%
%
%
%
\begin{remark}\label{rem:mean-convex}
Smooth minimal plurisubharmonic functions have already been used in minimal surface theory.
One direction was  to determine when a region $W$ in $\r^3$ is universal for minimal surfaces, 
in the sense that every connected properly immersed minimal surface $M$ in $\r^3$, with nonempty 
boundary and contained in $W$, is of parabolic conformal type; 
see \cite{Collin} and \cite[Sec.\ 3.2]{Meeks-Perez:2004} for a 
discussion of this question.

Minimal plurisubharmonic functions can also be used to define the class of 
{\em mean-convex domains} in $\r^3$. In the literature on minimal surfaces, 
a domain $D\subset \r^n$ with smooth boundary $bD$ (of class $\Cscr^{1,1}$ or better) 
is said to be (strongly) {\em mean-convex} if the sum of the principal curvatures of $bD$ 
from the interior side is nonnegative (resp.\ positive) at each point.  
(See e.g.\ the paper by Meeks and Yau \cite{Meeks-Yau:1982} and the references therein.) 
By Harvey and Lawson  \cite[Theorem 3.4]{Harvey-Lawson2013}, a domain $D\Subset \r^3$
is mean-convex if and only if it admits a smooth minimal strongly plurisubharmonic 
exhaustion function $\rho\colon D\to \r$. 
Mean-convex domains have mainly been studied as natural barriers for minimal hypersurfaces 
in view of the maximum principle. The smallest mean-convex barrier containing a compact set 
$K\subset \r^n$ (if it exists) is called the {\em mean-convex hull} of $K$; it coincides with 
the minimal hull $\widehat K_\Mgot$ of $K$ when $n=3$ (cf.\ Definition \ref{def:minimalhull}). 
Our definition of the minimal hull (by the maximum principle with respect to the class of minimal plurisubharmonic functions)  
applies to an arbitrary compact set and is in line with the standard notion 
of plurisubharmonic and holomorphic hull of a compact set in a complex manifold. 
The main technique used for finding the mean-convex hull of a given
compact set $K$  is the {\em mean curvature flow} of hypersurfaces (with $K$ as an obstacle) 
introduced by Brakke \cite{Brakke:1978}. For a discussion of this subject see for example the monographs 
by Bellettini \cite{Bellettini} and Colding and Minicozzi \cite{Colding-Minicozzi:2011}.
An interesting recent result in this direction is that given a compact set $K\subset \r^n$ for $n\le 7$
with boundary of class $\Cscr^{1,1}$, the boundary $b\wh K$ of its mean-convex hull $\widehat K_\Mgot$ is also of class 
$\Cscr^{1,1}$, and $b\widehat K_\Mgot \setminus K$ consists of minimal hypersurfaces (Spadaro \cite{Spadaro:2011}).
\end{remark}

%
%
%
%
\section{Null plurisubharmonic functions} 
\label{sec:null}
In this section we introduce the notion of a {\em null plurisubharmonic function}
on a domain in $\c^n$. One of our main results, Theorem \ref{th:nullpshminorant},
expresses the biggest null plurisubharmonic minorant of a given 
upper semicontinuous function $\phi$ on a domain in $\c^3$ 
as the envelope of the Poisson functional of $\phi$ on the family 
of null holomorphic discs.  This is the analogue of the Poletsky-Bu-Schachermayer 
theorem for plurisubharmonic functions  (cf.\ Theorem \ref{th:Pminorant}).
In the following section we shall use null plurisubharmonic functions 
to introduce the null hull of a compact set in $\c^n$ (Definition \ref{def:nullhull}).

Let $\Agot$ denote the conical quadric subvariety of $\c^n$ given by
\begin{equation} \label{eq:Agot}
		\Agot=\{z=(z_1,\cdots,z_n)\in\c^n\colon z_1^2+\cdots +z_n^2=0\}.
\end{equation}
This is called the {\em null quadric} in $\c^n$, and its elements are  {\em null vectors}. 
Note that $\Agot$ has the only singularity at the origin. We shall write $\Agot^*=\Agot\setminus \{0\}$. 
A holomorphic map $f:M\to\c^n$ from an open Riemann surface $M$ is said to be a
{\em null holomorphic map} if 
\begin{equation} \label{eq:Ndisc}
		f'_1(\zeta)^2 + f'_2(\zeta)^2  + \cdots + f'_n(\zeta)^2=0
\end{equation}
holds in any local holomorphic coordinate $\zeta$ on $M$; equivalently, if the derivative
$f'$ has range in $\Agot$. A null holomorphic map $f$ is a (null) immersion if and only if the derivative 
$f'$ has range in $\Agot^*=\Agot\setminus \{0\}$. We shall mainly be concerned with (closed) {\em null holomorphic discs}
in $\c^n$, or {\em null discs} for short; these are $\Cscr^1$ maps
$f\colon \overline \d=\{\zeta\in\c: |\zeta|\le 1\} \to\c^n$
from the closed disc $\cd$ which are holomorphic on the open disc 
$\d$ and satisfy the nullity condition (\ref{eq:Ndisc}).
We denote  by $\Ngot(\d,\Omega)$ the set of all immersed null holomorphic discs 
$\cd \to\Omega$ with range in a domain $\Omega\subset \c^n$.

%
%
%
%
\begin{definition} \label{def:npsh}
An upper semicontinuous function $u:\Omega\to\r\cup\{-\infty\}$ on a domain $\Omega\subset\c^n$ is
{\em null plurisubharmonic} if the restriction of $u$ to any affine complex line $L\subset \c^n$
directed by a null vector $\theta\in \Agot^*$ is subharmonic on $L\cap \Omega$,
where $\Agot$ is given by (\ref{eq:Agot}). The class of all such functions is denoted by $\NPsh(\Omega)$.
\end{definition}

Clearly  we have that $\Psh(\Omega) \subset \NPsh(\Omega)$, and the inclusion is proper as shown by the 
following example. (Further examples are furnished by Lemma \ref{lem:extend} below.)

%
%
%
%
\begin{example}\label{ex:npsh}
The function $u(z_1,z_2,z_3)=|z_1|^2+|z_2|^2-|z_3|^2$ is not pluri\-subharmonic on any 
open subset of $\c^3$. However, it is null  plurisubharmonic on $\c^3$ which is seen as follows.
Fix a null vector $z=(z_1,z_2,z_3)\in \Agot^*$.
Then $u(\zeta z)=|\zeta|^2 u(z)$ for $\zeta\in \c$, and we need to check that 
$u(z)\ge 0$. The equation (\ref{eq:Agot}) gives 
$z_3^2=-(z_1^2+z_2^3)$ and hence $|z_3|^2\le |z_1|^2+|z_2|^3$ by the triangle
inequality, so  $u(z)= |z_1|^2+|z_2|^2-|z_3|^2 \ge 0$.
(Note that $u$ vanishes on some null vectors; for example, on 
$z=(\imath \sqrt 2/2,\imath \sqrt 2/2,1)$.) 
For any point $a\in\c^3$ the function $u(a+\zeta z)$ differs from 
$u(\zeta z)$ only by a harmonic term in $\zeta$, so the restriction of 
$u$ to any affine complex line directed by a null vector is subharmonic.
\qed \end{example}

Given a $\Cscr^2$ function $u$ on a domain in $\c^n$ we denote by $\Lcal_u(x;\theta)$ the Levi form 
of $u$ at the point $x$ in the direction of the vector $\theta\in T_x\c^n$.

The next lemma summarizes some of the properties of  null plurisubharmonic functions which 
are analogous to the corresponding properties of plurisubharmonic functions.

%
%
%
%
\begin{proposition}
\label{lemma-nullpsh1}
Let $\Omega$ be a domain in $\c^n$.
	\begin{itemize}
		\item[(i)] If $u,v\in \NPsh(\Omega)$ and $c>0$, then $cu,\ u+v,\ \max\{u,v\}\in \NPsh(\Omega)$.
		\item[(ii)] If $u\in \Cscr^2(\Omega)$ then $u\in \NPsh(\Omega)$  if and 
		only if $\Lcal_u(x;\theta)\ge 0$ for every $x\in \Omega$ and $\theta\in \Agot$.
		\item[(iii)]  The limit of a decreasing sequence of null plurisubharmonic functions on $\Omega$
		 is a null plurisubharmonic function on  $\Omega$.
		\item[(iv)] If $\Fcal\subset \NPsh(\Omega)$ is a family  which is
		locally uniformly bounded above, then the upper semicontinuous regularization $v^*$ of the
		upper envelope $v(x)=\sup_{u\in\Fcal} u(x)$ is  null plurisubharmonic on $\Omega$.
		\item[(v)] If $u\in \NPsh(\c^n)$ is bounded above, then $u$ is constant.
		\item[(vi)] The Levi form of any null plurisubharmonic function of class $\Cscr^2$ 
		has at most one negative eigenvalue at each point.
	\end{itemize}
\end{proposition}

Properties (i)--(iv) follow in a standard way from the corresponding properties 
of subharmonic functions.
Property (v) follows from the fact that every bounded above subharmonic function on $\c$
is constant, and any two points of $\c^n$ can be connected by a finite chain of affine 
complex null lines. Part (vi) is seen by observing that every $2$-dimensional complex linear 
subspace of $\c^n$ intersects $\Agot^*$.

Part (ii) of Proposition \ref{lemma-nullpsh1} justifies the following definition.

%
%
%
%
\begin{definition}\label{def:SNPsh}
A function $u\in \Cscr^2(\Omega)$ on a domain $\Omega\subset \c^n$ is said to be
{\em strongly null plurisubharmonic} if $\Lcal_u(x;\theta)>0$ for every $x\in \Omega$
and $\theta\in\Agot^*$. 
\end{definition}

%
%
%
%
\begin{remark}\label{rem:NPsh}
Null plurisubharmonic functions are a natural substitute for plurisubharmonic functions
when considering only complex null curves (instead of all complex curves).
They are a special case of \emph{$\g$-plurisubharmonic functions}  of 
Harvey and Lawson \cite{Harvey-Lawson2012} who in a series of papers 
studied plurisubhamonicity in a more general  geometric context
(see  \cite{Harvey-Lawson2009,Harvey-Lawson2012,Harvey-Lawson2013},
among others). Let $X$ be a complex manifold endowed with a hermitian metric, 
and let $p\in \{1,\ldots,\dim X\}$ be an integer.
Let $G(p,X)$ denote the Grassmann bundle over $X$ whose fiber at a point $x\in X$ is the set 
of all complex $p$-dimensional subspaces of the tangent space $T_x X$.  Let $\g\subset G(p,X)$. 
A function $u\colon X\to \r$ of class $\Cscr^2$ is said to be \emph{$\g$-plurisubharmonic} 
\cite{Harvey-Lawson2012} if  for each $x\in X$ and $W\in \g_x$ the trace of the Levi form of $u$ at 
$x$ restricted to $W$ is nonnegative. If $\g=G(p,X)$, then $\g$-plurisubharmonic functions
are called {\em $p$-plurisubharmonic}; for $p=1$ we get the usual 
plurisubharmonic functions, while the case $p=\dim X$ corresponds to subharmonic functions. 
Taking $\g_z\subset G(1,\c^n)$ to be the set of null complex lines through the origin 
we get null plurisubharmonic functions.
\qed \end{remark}

One might ask whether there is any relationship between null plurisubharmonic functions
and $2$-plurisubharmonic functions (see Remark \ref{rem:NPsh} above), especially 
in light of the fact that the Levi form of a null plurisubharmonic function has at most 
one negative eigenvalue at each point. The following example shows that this is not the case.

%
%
%
%
\begin{example}
Let $\nu_\pm=\sqrt 2^{-1}(1,\pm \igot,0)$ and $\be_3=(0,0,1)$. These three vectors
form a unitary basis of $\c^3$, and $\nu_\pm$ are null vectors.  For $z=(z_1,z_2,z_3) \in\c^3$ let 
\[
	u(z_1\nu_+ +z_2\nu_- +z_3 \be_3)= \epsilon |z_1|^2+ b |z_2|^2-|z_3|^2.
\]
We claim that $u$ is null plurisubharmonic on $\c^3$ if $\epsilon>0$ is small enough 
and $b>0$ is big enough. As in Example \ref{ex:npsh} we only need to verify that
$u\ge 0$ on $\Agot$. Since $u$ is homogeneous, it suffices to check that $u>0$ on the compact
set $\Agot_1\subset \Agot$ consisting of unit null vectors. Let $\Sigma$ be the complex 
2-plane spanned by $\nu_+$ and $\be_3$.  Then $\Sigma \cap \Agot_1 =\{\E^{\imath t}\nu_+: t\in \r\}$.
The term $\epsilon |z_1|^2$ ensures positivity of $u$ on a neighborhood of $\Sigma \cap \Agot_1$
in $\Agot_1$. Since $b |z_2|^2$ is positive on $\Agot_1\setminus \Sigma$,
we see that $u>0$ on $\Agot_1$ if  $b$ is chosen big enough, so
$u$ is strongly null plurisubharmonic on $\c^3$. However, since the trace of the Levi form 
of $u$ on $\Sigma$ equals $\epsilon-1$, $u$ is not $2$-plurisubharmonic if $\epsilon<1$. 
\qed\end{example}

The next proposition gives some further  properties of null plurisubharmonic functions.
In particular,  we can approximate them by smooth null plurisubharmonic functions.

%
%
%
%
\begin{proposition}\label{lem:smoothing}
Let $\Omega$ be a domain in $\c^n$.
	\begin{itemize}
		\item[(i)] If $u\in \NPsh(\Omega)$ and $u\not\equiv-\infty$ on $\Omega$, 
		then $u\in L^1_{\rm loc}(\Omega)$.
		\item[(ii)] If $u\in \NPsh(\Omega)$ and $u\not\equiv-\infty$ on $\Omega$, 
		then $u$ can be  approximated 
		by smooth null plurisubharmonic functions on domains compactly contained in $\Omega$.
		\item[(iii)] If $u\in \NPsh(\Omega)$ and $f\in\Ngot(\d,\Omega)$ then 
		$u\circ f$ is subharmonic on $\cd$. More generally, if $f:M\to \Omega$ is a 
		null holomorphic curve then $u\circ f$ is  subharmonic on $M$.
		\item[(iv)] If $\Omega$ is a pseudoconvex Runge domain in $\c^n$, $u\in \NPsh(\Omega)$
		and $K$ is a compact set in $\Omega$, then there exists a function $v\in \NPsh(\c^n)$
		such that $v=u$ on $K$, $v$ is strongly plurisubharmonic on $\c^n\setminus \Omega$,
		and $v(z)>\max_K v$ for any point $z\in \c^n\setminus\Omega$.
	\end{itemize}
	\label{komp_psh_disc}
\end{proposition}

\begin{proof}
We adapt the usual proof for the plurisubharmonic case 
(see for example \cite[Lemma 4.11, Theorem 4.12, Theorem 4.13]{Range}
or Chapter 4 in \cite{Hormander-convexity}).  For simplicity of 
notation we consider the case $n=3$; the same proofs apply to any $n\ge 3$.

We begin by explaining the proof of part (i).
Since the cone $\Agot\subset\c^3$ (\ref{eq:Agot}) is not contained in any complex hyperplane of $\c^3$, 
there exist three $\c$-linearly independent vectors $\theta_1,\theta_2, \theta_3\in \Agot$. 
As in the standard case we assume that $u(p)>-\infty$ for some $p\in \Omega$, 
and we need to prove that for every $r>0$ such that 
$$
	D_p^r :=\{p+z_1\theta_1+z_2\theta_2+z_3\theta_3\colon |z_i|\le r,\ 1\le i\le 3\}\subset \Omega
$$
we have $u\in L^1(D_p^r)$. Since $u$ is bounded above on the compact set $D_p^r$, we only 
need to show that $\int_{D_p^r} u\,dV>-\infty$; the claim then follows as in the standard case.
 Fix $r>0$ as above. For any $z=(z_1,z_2,z_3)$ such that $|z_i|\in [0,r]$ for $i=1,2,3$
the restriction of $u$ to the disc
$\{p +\sum_{i=1}^{j-1}z_i \theta_i + \zeta \theta_j\colon |\zeta|\le r \}$ is subharmonic
for each $1\le j\le 3$. Applying the submean value property in each variable we obtain
\[
	u(p)\le \int_0^{2\pi}\!\!\! \int_0^{2\pi} \!\!\! \int_0^{2\pi} u(p+r_1 \E^{\imath t_1}
	\theta_1+r_2\E^{\imath t_2}\theta_2 +r_3\E^{\imath t_3}\theta_3)\, dt_1\,dt_2\, dt_3
\]
for all $r_i\in [0,r]$, $i=1,2,3$. Multiplying this inequality 
by $r_1 r_2 r_3 dr_1\, dr_2\, dr_3$ and integrating with respect to $r_i\in [0,r]$
for  $1\le i\le 3$ gives  $u(p)\le C\int_{D_p^r} u\,dV$, where the positive constant $C$ 
depends on the  choice of the $\theta_j$'s. This proves (i).

In part (ii) we proceed as in the usual proof for smoothing plurisubharmonic 
functions, convolving $u$ by a smooth approximate identity $\phi_t$ satisfying the property
$\phi_t(\sum_{i=1}^{3} z_i \theta_i)=\varphi_t(|z_1|,|z_2|,|z_3|)$.
We leave the obvious details to the reader.

For functions $u\in \Cscr^2(\Omega)$, property (iii) is an immediate consequence
of Proposition \ref{lemma-nullpsh1}-(ii). In the general case the same result 
follows from part (ii) (smoothing) and the fact that the limit of a decreasing sequence 
of subharmonic functions is subharmonic.

It remains to prove (iv).  Since $\Omega$ is pseudoconvex and Runge, the polynomial hull 
$\widehat K$ is contained in $\Omega$ \cite[Theorem 2.7.3, p.\ 53]{Hormander-SCV}. 
Pick an open set $U\Subset\Omega$ with $\widehat K\subset U$.
By \cite[Theorem 2.6.11, p.\ 48]{Hormander-SCV} 
there is a smooth plurisubharmonic exhaustion function $\rho\colon\c^n \to\r_+$ that vanishes on a
neighborhood of $\widehat K$ and is positive strongly plurisubharmonic on $\c^n \setminus U$.
Let $\chi\colon \c^n\to [0,1]$ be a smooth function that equals $1$
on $U$ and has support contained in $\Omega$. Then the function
$v=\chi u+ C\rho$ for big $C>0$ has the stated properties.
\end{proof}

In the sequel we shall use the following result of Alarc\'on and Forstneri\v c \cite[Lemma 3.1]{AF2013} 
which gives approximate solutions to a Riemann-Hilbert boundary value problem for null 
holomorphic discs in $\c^3$. As in \cite{DF2012}, we can add to their result 
the estimate \eqref{small-increase} of the average of a given function $u$ over 
the boundary of a suitably chosen null disc.  Note that 
the central null disc $f$ in Lemma \ref{prop:RHnull} below is arbitrary,  
but the null discs centered at boundary points 
$f(\zeta)$, $\zeta\in b\d=\t$, are linear round discs  in the same direction.

%
%
%
%
\begin{lemma}\label{prop:RHnull}
Let $f\colon\cd\to\c^3$ be a null holomorphic immersion, let 
$\theta\in\Agot^*$ be a null vector, and let $\mu\colon\t\to [0,\infty)$
be a continuous function. Define the map $g\colon\t\times\cd\to\c^3$ by
$g(\zeta,\xi)=f(\zeta)+\mu(\zeta)\xi\theta$. Given $\epsilon>0$ and $0<r<1$, there exist
a number $r'\in[r,1)$ and a null holomorphic immersion $h\colon\cd\to\c^3$ 
with $h(0)=f(0)$, satisfying the following properties:
\begin{itemize}
\item[(i)] $\dist(h(\zeta),g(\zeta,\t))<\epsilon$ for all $\zeta\in\t$,
\item[(ii)] $\dist(h(\rho \zeta),g(\zeta,\cd))<\epsilon$ for all $\zeta\in\t$ and all 
$\rho\in [r',1)$, and
\item[(iii)] $h$ is $\epsilon$-close to $f$ in ${\mathcal C}^1$ topology on 
$\{\zeta \in\c\colon|\zeta|\le r'\}$.
\end{itemize}
Furthermore, given an upper semicontinuous function $u\colon\c^3\to \r\cup\{-\infty\}$ 
and an arc $I\subset \t$, we may achieve, in addition to the above, that 
\begin{equation}
\label{small-increase}
	\int_{I} u\bigl( h(\E^{\imath t})\bigr) \, \frac{dt}{2\pi} \le 
 	\int^{2\pi}_0 \!\! \int_{I} u\bigl(g(\E^{\imath t},\E^{\imath s})\bigr) 
       \frac{dt}{2\pi} \frac{ds}{2\pi} + \epsilon.
\end{equation}
\end{lemma}

The following proposition is \cite[Lemma 2.1]{Edgar} and \cite[Proposition II.1]{Bu-Schachermayer}
for the case of null plurisubharmonic functions.

%
%
%
%
\begin{proposition}\label{prop:BS}
Let $\Omega$ be a domain in $\c^n$ and $\phi\colon \Omega \to \r\cup\{-\infty\}$ be 
an upper semicontinuous function.  Define $u_1=\phi$ and for $j > 1$
\begin{equation}
	u_j(z) = \inf \left\{\int^{2\pi}_0 u_{j-1} 
	\bigl( z+\E^{\imath t}\theta\bigr) \, \frac{dt}{2\pi}\right\}, \quad z\in \Omega,
\label{BSequation}
\end{equation}
where the infimum is taken over all vectors $\theta\in\Agot^*$ such that $\{z +\zeta\theta: |\zeta| \le 1 \}\subset \Omega$.
Then the functions $u_j$ are upper semicontinuous and decrease pointwise to the largest null plurisubharmonic function 
$u_\phi$ on $\Omega$ bounded above by $\phi$.
\label{propo_BS}
\end{proposition}

\begin{proof}
We follow the proof of \cite[Proposition II.1]{Bu-Schachermayer}.
We first show by induction that the functions $u_j$ (\ref{BSequation}) are upper semicontinuous. 
Assume that $j>1$ and that $u_{j-1}$ is upper semicontinuous 
(this holds when $j=2$). Choose a sequence $z_k$ in $\Omega$ converging to 
$z_0\in \Omega$ and a null vector $\theta \in\Agot^*$
such that $z_0+\cd\theta =\{z_0+ \zeta\theta: |\zeta| \le 1\}\subset \Omega$.
For $k_0\in \n$ big enough the set $U=\bigcup_{k=k_0}^\infty (z_k+\cd\theta)$ is relatively compact in $\Omega$.  
Then $u_{j-1}$ is bounded above on $U$ and satisfies 
$u_{j-1}(z_0+\zeta \theta)\ge \limsup_{k\to \infty} u_{j-1}(z_k+\zeta \theta)$ 
for all $\zeta \in \cd$. Fatou's lemma implies
\[
	\int^{2\pi}_0 u_{j-1} \bigl( z_0+\E^{\imath t}\theta\bigr) \, \frac{dt}{2\pi}\ge
	\limsup_{k\to \infty}\int^{2\pi}_0 u_{j-1} 
	\bigl( z_k+\E^{\imath t}\theta\bigr) \, \frac{dt}{2\pi}\ge \limsup_{k\to \infty} u_j(z_k).
\]
Taking the infimum over all null vectors $\theta$ we get $u_j(z_0)\ge\limsup_{k\to \infty} u_j(z_k)$.
Therefore $u_j$ is upper semicontinuous which concludes the inductive step.
The sequence $u_j$ is obviously pointwise decreasing, so the limit function $u_\phi$ is also upper semicontinuous. 

To show that $u_\phi$ is null plurisubharmonic, pick a point $z\in \Omega$ and a 
null vector $\theta \in\Agot^*$ such that $z+\cd\,\theta\subset \Omega$. 
It follows from Beppo Levi monotone convergence theorem that
\[
	u_\phi(z)=\lim_{j\to\infty} u_j(z)\le \lim_{j\to\infty}\int^{2\pi}_0 u_{j-1} 
	\bigl( z+\E^{\imath t}\theta\bigr) \, \frac{dt}{2\pi}=
	\int^{2\pi}_0 u_\phi \bigl(z+\E^{\imath t}\theta\bigr) \, \frac{dt}{2\pi}.
\]

It remains to prove that $u_\phi$ is the largest null plurisubharmonic function dominated by $\phi$. 
Choose a null plurisubharmonic function $v\le \phi$. We show by induction that $v\le u_j$ 
for every $j\in\n$; clearly  this will imply that $v\le u_\phi$.
Suppose that $v\le u_{j}$ for some $j\in\n$; this trivially holds for $j=1$ since $u_1=\phi$. 
For every point $z\in\Omega$ and every null vector 
$\theta \in\Agot^*$ such that $z+\cd\,\theta\subset \Omega$ we then have that
\[
	v(z)\le \int^{2\pi}_0 v \bigl( z+\E^{\imath t}\theta\bigr) \, \frac{dt}{2\pi}
	\le \int^{2\pi}_0 u_{j} \bigl( z+\E^{\imath t}\theta\bigr) \, \frac{dt}{2\pi}.
\]
Taking infimum over  all $\theta$ we get $v(z)\le u_{j+1}(z)$ which concludes the inductive step.
\end{proof}

Given a point $z \in \Omega$ we define $\Ngot(\d,\Omega,z)=\{f \in \Ngot(\d,\Omega): f(0)=z\}$. 
We are now ready to prove the following central result of this section.

%
%
%
%
\begin{theorem} [\bf Null plurisubharmonic minorant]
\label{th:nullpshminorant} 
Let $\phi\colon \Omega\to \r\cup\{-\infty\}$ be an upper semicontinuous 
function on a domain $\Omega\subset\c^3$. Then the function 
\begin{equation}
\label{eq:nullPoisson-funct}
   u(z) = \inf \Big\{\int^{2\pi}_0 \phi(f(\E^{\imath t}))\, 
   \frac{dt}{2\pi} \colon \ f\in \Ngot(\d,\Omega,z) \Big\}, \quad z\in \Omega,
\end{equation}
is null plurisubharmonic on $\Omega$ or identically $-\infty$; 
moreover, $u$ is the supremum of the null plurisubharmonic 
functions on $\Omega$ which are not greater than $\phi$. 
\end{theorem}

\begin{proof} 
Proposition \ref{propo_BS} furnishes a decreasing sequence of upper semicontinuous 
functions $u_n$  on $\Omega$ which converges pointwise to the largest  null 
plurisubharmonic function $u_\phi$ on $\Omega$ dominated by $\phi$. 
To conclude the proof we need to show that
\[ 
	u_\phi(z)=\inf \Big\{\int^{2\pi}_0\phi(f(\E^{\imath t}))\, 
       \frac{dt}{2\pi} \colon \ f\in \Ngot(\d,\Omega,z) \Big\}, \quad z\in \Omega.
\]
We denote the right hand side of the above equation by $u(z)$.
   
Since $u_\phi$ is a null plurisubharmonic function on $\Omega$
dominated by $\phi$, Proposition \ref{komp_psh_disc} gives the following estimate 
for any $f\in \Ngot(\d,\Omega,z)$:
\[
	u_\phi(z)\le \int^{2\pi}_0 u_\phi(f(\E^{\imath t}))\, \frac{dt}{2\pi}
	\le \int^{2\pi}_0 \phi(f(\E^{\imath t}))\,    \frac{dt}{2\pi}
\]
Taking the infimum over all such $f$ we obtain $u_\phi\le  u$ on $\Omega$.

To prove the reverse inequality, fix a point $z\in \Omega$ and choose a number $\epsilon>0$.
Since $u_n(z)$ decreases to $u_\phi(z)$ as $n\to\infty$,  there is a positive integer $n$ so large that 
\begin{equation}\label{eq:0step}
	u_\phi(z)\le u_n(z)< u_\phi(z)+ \epsilon .
\end{equation}
By (\ref{BSequation}) there exists a null vector $\theta\in \Agot^*$
such that the null disc $f_{n-1}(\zeta)= z+\zeta \theta$ $(\zeta\in\cd)$ lies in $\Omega$ and satisfies
\begin{equation}\label{eq:1step}
	\int^{2\pi}_0 u_{n-1} \left(f_{n-1}(\E^{\imath t})\right) \, \frac{dt}{2\pi}
	\le u_n(z)+\frac \epsilon {n}.
\end{equation}
Fix a point $\E^{\imath t_0} \in \t$. By the definition of $u_{n-1}$ (\ref{BSequation}) there exists 
a null vector $\theta_{t_0}\in \Agot^*$
such that the null disc $\cd\ni \zeta \mapsto f_{n-1}(\E^{\imath t_0})+\zeta \theta_{t_0}$ 
lies in $\Omega$ and satisfies
\begin{equation}
\label{eq:suboptimal}
	\int^{2\pi}_0 u_{n-2} \left(f_{n-1}(\E^{\imath t_0})+
	\E^{\imath t}\theta_{t_0}\right) \, 
	\frac{dt}{2\pi}\le u_{n-1}(f_{n-1}(\E^{\imath t_0}))+\frac \epsilon {4n}.
\end{equation}
Setting $g_{n-1}(\E^{\imath s},\zeta)=f_{n-1}(\E^{\imath s})+\zeta\theta_{t_0}$,  
it follows from (\ref{eq:suboptimal}) that  there is a small arc $I\subset \t$ 
around the point $\E^{\imath t_0}$ such that
\[
      \int^{2\pi}_0 \!\! \int_{I} 
       u_{n-2}\bigl(g_{n-1}(\E^{\imath s},\E^{\imath t})\bigr) \, \frac{ds}{2\pi} \frac{d t}{2\pi} 
       \le 
       \int_I u_{n-1}\bigl(f_{n-1}(\E^{\imath s})\bigr)\,\frac{ds}{2\pi} + 
       \frac{|I|}{2\pi}\frac{\epsilon}{3n}.     
\]

By repeating this construction at other points of $\t$
we find finitely many closed arcs $I''_j \subset \t$ $(j=1,\ldots,l)$
such that  $\bigcup_{j=1}^l I''_j=\t$.
The function $u_{n-2}$ is bounded above by some constant $M$ on 
$\bigcup_{j=1}^l\bigcup_{\zeta\in I''_j} (f_{n-1}(\zeta)+\cd\, \theta_{t_j})$. 
We can choose smaller arcs $I_j\Subset I'_j\Subset I''_j$ $(j=1,\ldots,l)$
such that $\overline I'_j \cap \overline I'_k=\emptyset$ if $j\ne k$ and
the set $E=\t\setminus \bigcup_{j=1}^l I_j$ has 
arbitrarily small measure $|E|$, for example less than $\frac{\epsilon}{2nM}$, and 
smooth families of affine null discs $g_{n-1}(\zeta,\xi)=f_{n-1}(\zeta)+\xi \theta_{t_j}$ 
for $(\zeta,\xi )\in I'_j\times \cd$ such that
\begin{equation}
\label{est1}
   \int^{2\pi}_0 \!\! \int_{I_j} 
       u_{n-2}\bigl(g_{n-1}(\E^{\imath s},\E^{\imath t})\bigr) \, 
      \frac{ds}{2\pi} \frac{dt}{2\pi} 
  \le \int_{I_j} u_{n-1}\bigl(f_{n-1}(\E^{\imath s})\bigr)\,
   \frac{ds}{2\pi}
    + \frac{|I|}{2\pi}\frac{\epsilon}{2n}.
\end{equation}
Let $\chi\colon\t\to [0,1]$ be a smooth function such that $\chi\equiv 1$ on $\bigcup_{j=1}^l I_j$ 
and $\chi\equiv 0$ on a neighborhood of the set $\t\setminus \bigcup_{j=1}^l I'_j$.
Define the map $h_{n-1}\colon\t\times\cd\to\c^3$ by
\[
	h_{n-1}(\zeta,\xi)=g_{n-1}(\zeta,\chi(\zeta) \xi), \quad  (\zeta,\xi)\in \t\times \cd.
\]
By Lemma \ref{prop:RHnull} we get a null disc $f_{n-2}$ centered at $z$ and satisfying
\[
	\int^{2\pi}_0 u_{n-2} \left(f_{n-2}(\E^{\imath t})\right) \, \frac{dt}{2\pi}
	\le  \int^{2\pi}_0 u_{n-1}\bigl(f_{n-1}(\E^{\imath t})\bigr)\, \frac{dt}{2\pi}
        + \frac{\epsilon}{n}.
\]
(We apply Lemma \ref{prop:RHnull} $l$ times, once for each of the segments $I_1,\ldots, I_l$.)

Repeating this procedure we get null discs
$f_{1},f_2,\ldots,f_{n-1}$ in $\Omega$,  centered at $z$, such that
for $k=1,\ldots,n-2$ we have 
\[
	\int^{2\pi}_0 u_{k} \left(f_{k}(\E^{\imath t})\right) \, \frac{dt}{2\pi}\le  
	\int^{2\pi}_0 u_{k+1}\bigl(f_{k+1}(\E^{\imath t})\bigr)\, \frac{dt}{2\pi}
    + \frac{\epsilon}{n}.
\]
Since $u_1=\phi$, we get  by \eqref{eq:0step}  and \eqref{eq:1step} that
\[
	\int^{2\pi}_0 \! \phi \left(f_{1}(\E^{\imath t})\right) \, \frac{dt}{2\pi}\le
	\int^{2\pi}_0 u_{n-1}\bigl(f_{n-1}(\E^{\imath t})\bigr)\, \frac{dt}{2\pi}
        + \frac{(n-2)\epsilon}{n}\le u_n(z)+ \epsilon\le  u_\phi(z)+ 2\epsilon.
\]
Therefore $u (z)\le u_\phi(z)+ 2\epsilon$. Since this holds for any 
given $\epsilon >0$, we get $u(z)\le u_\phi(z)$. 
This completes the proof of Theorem \ref{th:nullpshminorant}.
\end{proof}

%
%
%
%

\section{Null hulls of compact sets in $\c^3$}
\label{sec:nullhull}

In the section we introduce the {\em null hull} of a compact set in 
$\c^n$. This is a special case of {\em $\mathbb G$-convex hulls} introduced by 
Harvey and Lawson in \cite[Definition 4.3, p.\ 2434]{Harvey-Lawson2012}.

%
%
%
%
\begin{definition} \label{def:nullhull}
Let $K$ be a compact set in $\c^n$ $(n\ge 3)$. The {\em null hull}  of $K$ is defined by
\begin{equation}\label{eq:nullhull}
		\wh K_{\Ngot} =\{z\in \c^n: v(z)\le \max_K v\ \ \ \forall v\in \NPsh(\c^n)\}.
\end{equation}  
\end{definition}

The maximum principle for subharmonic functions implies 
that for any bounded null holomorphic curve $A\subset \c^n$ with boundary $bA\subset K$ 
we have $A\subset \wh K_{\Ngot}$. 

Since $\Psh(\c^n)\subset \NPsh(\c^n)$, we have 
\begin{equation}\label{eq:comparehulls}
	K\subset   \wh K_{\Ngot}\subset  \wh K  \subset \Co(K).
\end{equation}
The polynomial hull $\wh K$ is  rarely equal to the convex hull $\Co(K)$ of $K$.
The following example shows that in general we also have $\wh K_{\Ngot}\ne \wh K$.

%
%
%
%
\begin{example}\label{ex:hulls}
Let $K=\{(0,0,\E^{\imath t}):t\in\r\}$, the unit circle in the $z_3$-axis.
Clearly $\wh K=\{(0,0,\zeta): |\zeta|\le 1\}$. However, since the function 
$u(z_1,z_2,z_3)=|z_1|^2+|z_2|^2-|z_3|^2$ is null  plurisubharmonic 
(cf.\ Example \ref{ex:npsh}) and it equals $-|z_3|^2$ on the coordinate 
axis $\{(0,0)\}\times\c$, we see that $\wh K_{\Ngot}=K$.
(See Theorem \ref{th:NHcurves} for a more general result.)
\end{example}

The following lemma is an immediate consequence of the inclusion $\wh K_{\Ngot}\subset  \wh K$
(\ref{eq:comparehulls}), the standard fact that $\wh K\subset\Omega$ for any compact 
set $K$ in a pseudoconvex Runge domain $\Omega\subset\c^n$ 
\cite[Theorem 2.7.3]{Hormander-SCV}, and of Proposition \ref{lem:smoothing}-(iv)
which shows that the restriction map $\NPsh(\c^n) \to \NPsh(\Omega)$ has 
dense image.

%
%
%
%
\begin{lemma}\label{lem:Nhull-Runge}
If $\Omega$ is pseudoconvex Runge domain in $\c^n$, then for any compact
set $K$ in $\Omega$ we have $\wh K_{\Ngot}\subset\Omega$ and 
\begin{equation}\label{eq:nullhull-psc}
		\wh K_{\Ngot} =\{z\in \Omega: v(z)\le \max_K v\ \ \ \forall v\in \NPsh(\Omega)\}.
\end{equation}  
\end{lemma}

The following result is proved by following the standard case of plurisubharmonic 
functions (see \cite[Theorem 2.6.11, p.\ 48]{Hormander-SCV} or
\cite[Theorem 1.3.8, p.\ 25]{Stout-convexity}).

\begin{proposition}\label{prop:exhaustion}
Given a compact null convex set $K=\wh K_\Ngot\subset \c^n$ and an open set $U\supset K$, 
there exists a smooth null plurisubharmonic exhaustion function $\rho\colon\c^n\to\r_+$ such that
$\rho=0$ on a neighborhood of $K$ and $\rho$ is positive strongly null plurisubharmonic 
on $\c^n\setminus U$.
\end{proposition}

Our next result, which is essentially a corollary to Theorem \ref{th:nullpshminorant},
characterizes the null hull of a compact set in $\c^3$ by Poletsky sequences of null holomorphic discs.

%
%
%
%
\begin{corollary}\label{cor:DD}
Let $K$ be a compact set in $\c^3$ and let  $\Omega\subset \c^3$ be a
bounded pseudoconvex Runge domain containing $K$.
A point $p\in \Omega$ belongs to the null hull $\widehat K_\Ngot$ of $K$ 
if and only if there exists a sequence of null holomorphic discs $f_j\in \Ngot(\d,\Omega,p)$ 
such that 
\begin{equation} \label{eq:Ndiscs}
	\big| \{t\in [0,2\pi]: \dist(f_j(e^{\imath t}),K)<1/j\}  \big| \ge 2\pi -1/j,
	\qquad j=1,2,\ldots.
\end{equation}
\end{corollary}

\begin{proof}
We follow the case of polynomial hulls (cf.\ Theorem \ref{th:Poletsky}).
Since $\Omega$ is pseudoconvex and Runge in $\c^3$, we have 
$\widehat K_\Ngot \subset \widehat K\subset \Omega$, where the first inclusion
uses (\ref{eq:comparehulls}) and the second one is a standard result
\cite[Theorem 2.7.3]{Hormander-SCV}. 
Assume that for some point $p\in \Omega$ there exists a sequence $f_j\in \Ngot(\d,\Omega,p)$
satisfying (\ref{eq:Ndiscs}). Pick $u \in \NPsh(\c^3)$. 
Let $U_j=\{z\in\c^3: \dist(z,K)< 1/j\}$, $M_j=\sup_{U_j} u$, $M=\sup_{\Omega} u$, and 
$E_j= \{t\in [0,2\pi] \colon f_j(\E^{\imath t}) \notin U_j\}$.
Then $|E_j|\le 1/j$ by (\ref{eq:Ndiscs}). Since $u\circ f_j$ is subharmonic, we have
\[
	u(p) = u(f_j(0)) \le \int_{E_j} u(f_j(\E^{\imath t}))\,  \frac{dt}{2\pi} +
	\int_{[0,2\pi]\setminus E_j} u(f_j(\E^{\imath t}))\,  \frac{dt}{2\pi}
	\le M/j + M_j.
\]
Passing to the limit as $j\to\infty$ gives $u(p)\le \sup_K u$.
This shows that $p\in \widehat K_\Ngot$. 

To prove the converse, pick an open set $U$ in $\c^3$ with $K\subset U\Subset \Omega$. 
The function $\phi\colon \Omega\to\r$, which equals $-1$ on $U$ and equals $0$ 
on $\Omega\setminus U$, is upper semicontinuous. Let $u\in \NPsh(\Omega)$ be the associated 
extremal null plurisubharmonic function  (\ref{eq:nullPoisson-funct}). 
Then clearly  $-1\le u\le 0$ on $\Omega$, and $u=-1$ on $K$.
Hence Lemma \ref{lem:Nhull-Runge} implies
that $u(p)=-1$ for any point $p \in \wh K_{\Ngot}$. Fix such $p$ and pick
a number $\epsilon>0$. Theorem \ref{th:nullpshminorant} furnishes a null
disc $f\in \Ngot(\d,\Omega,p)$ such that 
$ 
	 \int^{2\pi}_0 \phi(f(\E^{\imath t}))\,  \frac{dt}{2\pi}< -1+ \epsilon/2\pi.
$ 
By the definition of $\phi$ this implies that
$|\{t\in [0,2\pi] \colon f(\E^{\imath t}) \in U\}| \ge 2\pi -\epsilon$. 
Applying this with the sequence of sets $U_j = \{z\in \c^3: \dist(z,K)<1/j\}$ and 
numbers $\epsilon_j=1/j$ gives Corollary \ref{cor:DD}.
\end{proof}

If we take $\Omega=\c^3$ then the first part of the proof of Corollary \ref{cor:DD}   fails since the sequence 
of discs $f_j$ need not be bounded, but the converse part still holds and gives the following  
observation which was pointed out by Rosay \cite{Rosay1,Rosay2} 
in the case of holomorphic discs -- we can find null discs with a given center $p$ 
and with most of their  boundaries squeezed in any given open set, possibly very small and 
far away from $p$. 

%
%
%
%
\begin{corollary}\label{cor:strangediscs}
Given a point $p\in \c^3$ and a nonempty open set $B\subset \c^3$, there exists 
a sequence of null holomorphic discs $f_j\in \Ngot(\d,\c^3,p)$ such that 
\begin{equation} \label{eq:strangediscs}
	\big| \{t\in [0,2\pi]: f_j(e^{\imath t}))\in B\}  \big| \ge 2\pi -1/j,
	\qquad j=1,2,\ldots.
\end{equation}
\end{corollary}

\begin{proof}
Let $\phi\colon \c^3\to\r$ equal $-1$ on  $B$ and equal $0$ on
$\c^3\setminus B$, so $\phi$ is upper semicontinuous.
Let $u\in \NPsh(\c^3)$ be the associated extremal null plurisubharmonic function 
defined by (\ref{eq:nullPoisson-funct}). Then $u\le 0$ on $\c^3$ and 
$u=-1$ on $B$. It follows from Proposition \ref{lemma-nullpsh1}-(v) that 
$u$ is constant on $\c^3$, so $u(p)=-1$.  
The same argument as in the proof of Corollary \ref{cor:DD} gives a sequence $f_j\in \Ngot(\d,\c^3,p)$
satisfying (\ref{eq:strangediscs}).
\end{proof}

As a consequence of the Alexander-Stolzenberg-Wermer  theorem on polynomial
hulls of compact sets of finite length in $\c^n$ 
\cite{Alexander1971,Stolzenberg1966,Wermer1958},
we now obtain the following description of null hulls of such set.  

%
%
%
%
\begin{theorem}  \label{th:NHcurves}
If $\Gamma$ is a compact set in $\c^n$ $(n\ge 3)$ contained in a connected compact set of finite linear measure, then $\widehat \Gamma_\Ngot\setminus \Gamma$ is a complex null curve (or empty).
\end{theorem}

\begin{proof} 
According to Alexander \cite{Alexander1971}, the set 
$V=\widehat \Gamma \setminus \Gamma$ is a (possibly empty) 
closed bounded one-dimensional complex subvariety of $\c^n\setminus \Gamma$ with 
$\overline V\setminus V\subset \Gamma$. 
Let $V=\bigcup_j V_j$ be a decomposition of $V$ into irreducible components,
and let $A$ denote the union of all components $V_j$ which are null curves.
Then $A$ is a bounded complex null curve and $\overline A\setminus A\subset\Gamma$.
Clearly $A\subset\widehat \Gamma_\Ngot$ by the maximum principle
for subharmonic functions.

We claim that $\widehat \Gamma_\Ngot=\Gamma\cup A$; this will prove the theorem. 
To this end we will show that for any $p \in V\setminus A$ there exists a 
function $\phi\in \NPsh(\c^n)$ such that $\phi(p)>\max_\Gamma \phi$. 

Let $B$ denote the union of all irreducible component $V_i$ of $V$ which are not
contained in $A$  (i.e., which are not null curves). 
Then $B$ is a bounded $1$-dimensional complex subvariety 
of $\c^n\setminus \Gamma$ and $\overline B\setminus B\subset\Gamma$. 
Let $C(B)$ denote the union of the singular locus $B_\sing$ and the set of
points $z\in B_\reg$  such that the tangent line $T_z B\subset \Agot$ is a null line.
Then $C(B)$ is a closed discrete subset of $B$ which clusters only on $\Gamma$.
(To see that $C(B)$ cannot cluster at a singular point $z\in B_\sing$, choose 
a local irreducible component $B_j$ of $B$ at $z$ and parametrize $B_i$ 
locally near $z$ by a nonconstant holomorphic map $f:\d\to B_i$ with $f(0)=z$.
Clearly the set $\{\zeta\in \d:f'(\zeta)\in \Agot\}$ is either all of $\d$ or else
is discrete in $\d$; the first case is impossible by the definition of $B$.) Hence the set 
\begin{equation}\label{eq:K} 
	K=\Gamma\cup A\cup C(B)  \subset \widehat \Gamma
\end{equation} 
is compact. Fix a point $p\in B\setminus A$; then either $p\notin K$ or $p$ is an isolated
point of $K$. Choose a pair of bounded open sets $U_0, U_1\Subset  \c^n$ such that
\[
		\overline U_0\cap \overline U_1= \emptyset, \quad
 		K\setminus \{p\} \subset U_0,\quad p\in U_1.
\]
Pick a smooth function $h\colon \c^n\to [0,1]$ such that  $h=0$ on $U_0$ 
(in particular, $h=0$ on $\Gamma$) and $h=1$ on $U_1$. 
Choose $\epsilon>0$ small enough such that the function 
$\tilde h(z) = h(z)+\epsilon |z-p|^2$
($z\in\c^n$) satisfies $\max_\Gamma  \tilde h<1$. Note that $\tilde h(p)=1$
and $\tilde h$ is strongly plurisubharmonic on $U_0\cup U_1$ 
(since $h$ is locally constant there and $z\mapsto |z-p|^2$ is strongly plurisubharmonic
on $\c^n$). Pick a smooth function $\chi\colon \c^n\to [0,1]$ such that $\chi=0$ 
on a small open neighborhood of the set $K$ (\ref{eq:K}) 
and $\chi>0$ on $\c^n\setminus (U_0\cup U_1)$. 
Since $V$ is a closed complex curve in $\c^n\setminus \Gamma$, there exists 
a plurisubharmonic function $\phi\ge 0$ on an open neighborhood
of $V$ in $\c^n$ that vanishes quadratically on $V$ and satisfies
\begin{equation}\label{eq:phi}
	\Lcal_\phi(z;\theta)>0\ \  
	\text{for\ all}\ z\in V_\reg\ \  \text{and}\ \theta\in  T_z\c^n\setminus T_z V.
\end{equation}
(Such $\phi$ can be chosen of the form $\phi=\sum_k |f_k|^2$ where 
$\{f_k\}$ is a collection of holomorphic defining functions for $V$ 
on a Stein neighborhood of $V$ in $\c^n$.)
We claim that for $C>0$ chosen big enough the function
\[
	v = \tilde h + C\chi \phi
\] 
is null plurisubharmonic on an open neighborhood of 
$\Gamma \cup V=\widehat \Gamma$ in $\c^n$. (Even though 
$\phi$ is only defined near $V$, the multiplier $\chi$ vanishes
near $\Gamma$ and hence we extend the product $\chi \phi$ as zero 
on a neighborhood of $\Gamma$.)
We have $v=\tilde h$ on a neighborhood $D \Subset U_0\cup U_1$ of $K$, 
so $v$ is strongly plurisubharmonic there. Suppose now that $z\in V\setminus D$. 
Then $z$ is a regular point of $V$ and the tangent line $T_z V$ is not a null line,
so $T_z V \cap \Agot=\{0\}$. 
It follows from (\ref{eq:phi}) that $\Lcal_\phi(z;\theta)>0$ for every $\theta\in \Agot^*$.
Since the set $V\setminus D$ is compact, we can ensure by choosing $C>0$
big enough that $\Lcal_v(z;\theta)>0$ for every $z\in V\setminus D$ and
$\theta\in \Agot^*$, which proves the claim. Observe also that $v(p)=\tilde h(p)=1$
and $\max_\Gamma v=\max_\Gamma \tilde h <1$.

Since $\widehat \Gamma$ is polynomially convex,
Proposition \ref{lem:smoothing}-(iv) furnishes a  null plurisubharmonic
function  $\phi \in \NPsh(\c^n)$ which agrees with $v$ near
$\widehat \Gamma$. Then $1=\phi(p)>\max_\Gamma \phi$ and hence
$p\notin\widehat\Gamma_\Ngot$. This completes the proof.
\end{proof}

In the course of proof of Theorem \ref{th:NHcurves} we have actually shown
the following result.

\begin{lemma}\label{lem:extend}
Let $V$ be a smooth locally closed complex curve in $\c^n$ $(n\ge 3)$ whose tangent 
line $T_z V$ is not a null line for any $z\in V$. Then for every  $h\in \Cscr^2(V)$
there exists an open neighborhood $\Omega\subset\c^n$ of $V$ and 
a strongly null plurisubharmonic function $v\in \NPsh(\Omega)$  such that
$v|_V=h$.
\end{lemma}

Theorem \ref{th:NHcurves} also implies the following corollary; clearly one can formulate
more general statements in this direction.

\begin{corollary}\label{cor:NHcurves}
Assume that $\Gamma$ is a rectifiable connected Jordan curve in $\c^n$ $(n\ge 3)$ which contains
an embedded arc $I\subset \c^n$ of class $\Cscr^r$ for some $r>1$.
If there exists a point $p\in I$ such that the tangent line $T_p I$ does not belong to 
the null quadric $\Agot$, then $\widehat \Gamma_\Ngot=\Gamma$.
\end{corollary}

\begin{proof}
If $\widehat \Gamma\ne \Gamma$ then $V=\widehat \Gamma\setminus \Gamma$
is a connected complex curve with boundary $bV=\Gamma$ by Alexander's theorem \cite{Alexander1971}. 
Since the arc $I\subset bV$ is of class $\Cscr^r$ with $r>1$, the union
$V\cup I$ is a $\Cscr^1$ variety with boundary  along  $I$ 
(see Chirka \cite{Chirka1989}). If $T_p I$ does not belong to $\Agot$ 
for some $p\in I$ then $V$ is not a null curve; hence the conclusion 
follows from Theorem \ref{th:NHcurves}.
\end{proof}

%
%
%
%
%
\section{Minimal plurisubharmonic functions and minimal hulls}
\label{sec:minimal}
In this section we study {\em minimal plurisubharmonic functions} on domains 
in $\r^n$ and {\em minimal hulls} of compact subsets of $\r^n$.
The central results, Theorem  \ref{th:minpshminorant}
and Corollary \ref{cor:minullhull}, are only proved for $n=3$
since they rely on the disc formula in Theorem \ref{th:nullpshminorant}.

%
%
\begin{definition}\label{def:minimalpsh}
An upper semicontinuous function $u$ on a domain $\omega\subset\r^n$ $(n\ge 3)$ is
{\em minimal plurisubharmonic} if the restriction of $u$ to any affine $2$-dimensional 
plane $L\subset \r^n$ is subharmonic on $L\cap \omega$ (in any isothermal
linear coordinates on $L$). The set of all such functions will be denoted by $\MPsh(\omega)$.
A function $u\in \Cscr^2(\omega)$ is {\em minimal strongly plurisubharmonic} 
if the restriction of $u$ to any affine $2$-dimensional plane $L\subset \r^n$ 
is strongly subharmonic on $L\cap \omega$.
\end{definition}

Linear coordinates $(y_1,y_2)$ on a 2-plane $L\subset \r^n$ are isothermal if $L$ is parametrized by
$x=a+y_1 v_1 + y_2 v_2$ where $a\in L$ and $v_1,v_2\in\rø^3 $ is a pair of orthogonal vectors of equal length.
Note that if $u\in\Cscr^2(\omega)$ is minimal plurisubharmonic then $u(x)+\epsilon |x|^2$
is minimal strongly plurisubharmonic on $\omega$ for every $\epsilon>0$.
(Here $|x|^2=\sum_{i=1}^n x_i^2$.) 

%
%
%
%
\begin{remark} \label{rem:2psh}
Smooth minimal plurisubharmonic functions are exactly {\em $2$-plurisubharmonic functions}
studied by Harvey and Lawson in \cite{Harvey-Lawson2013} (see in particular
Definition 2.2,  Proposition 2.3 and Definition 4.1 in \cite[p.\ 153]{Harvey-Lawson2013}). 
They are characterized by the property that the sum of the two smallest eigenvalues of the
Hessian is nonnegative at each point.  Here we adopt a more suggestive terminology  to 
emphasize their relationship with minimal surfaces. 
\end{remark}

There is a close connection between minimal plurisubharmonic functions on $\r^n$ 
and null plurisubharmonic functions on $\c^n$.

%
%
%
%
\begin{lemma} \label{lem:minimal}
Let $\omega\subset\r^n$ and $\Omega=\Tcal_\omega =\omega\times \imath \r^n\subset\c^n$.
\begin{itemize}
\item If $u\in\MPsh(\omega)$ then the function $U(x+\imath y)=u(x)$ $(x+\imath y\in \Omega)$ 
is null plurisubharmonic on the tube $\Omega=\Tcal_\omega$.
\item Conversely, if $U\in \NPsh(\Omega)$ is independent of
the variable $y=\Im z$ then the function $u(x)=U(x+\imath 0)$ $(x\in \omega)$
is minimal plurisubharmonic on $\omega$.
\end{itemize}
\end{lemma}

\begin{proof}
Recall that the real and the imaginary part of a holomorphic null disc $f\in \Ngot(\d,\c^n)$ 
are conformal minimal discs in $\r^n$; conversely, every conformal minimal disc in 
$\r^n$ is the real part of a holomorphic null disc in $\c^n$. Since 
$U\circ f=u\circ \Re f$ for all $f\in \Ngot(\d,\Omega)$,  the lemma follows.
\end{proof}

Lemma \ref{lem:minimal}  shows that properties (i), (iii)-(v) in 
Lemma \ref{lemma-nullpsh1}, and (i)-(iii) in Proposition \ref{komp_psh_disc}, 
of null plurisubharmonic functions  descend to the corresponding properties of 
minimal plurisubharmonic functions.  In particular, we have the following result.

%
%
%
%
\begin{lemma}\label{lem:characterizationMS} 
An upper semicontinuous function $u$ on a domain $\omega\subset\r^n$ is minimal plurisubharmonic 
if and only if for every conformal minimal immersion $f:\d\to \omega$ the composition
$u\circ f$ is subharmonic on $\d$. More generally, the restriction $u|_M$ of a minimal plurisubharmonic function 
to any minimal  $2$-dimensional submanifold $M$ is subharmonic in any isothermal 
coordinates on $M$. 
\end{lemma}

A precise expression for the Laplacian of $u\circ f$, where $f\colon\d \to\omega$ is a 
conformal minimal immersion, is given by (\ref{eq:ddcuf}) below.

For any open set $\omega\subset \r^n$ $(n\ge 3)$ we denote by $\Mgot(\d,\omega)$
the set of all conformal minimal immersions $f\colon \overline \d\to \omega$.
Given a point $x\in \omega$  we set 
\[
	\Mgot(\d,\omega,x)= \{f\in \Mgot(\d,\omega): f(0)=x\}.
\] 

The following result gives an effective way of constructing minimal plurisubharmonic functions on domains in $\r^3$. 

%
%
%
%
\begin{theorem} 
\label{th:minpshminorant} 
Let $\omega$ be a domain in $\r^3$
and let $\phi\colon \omega\to \r\cup\{-\infty\}$ be an upper semicontinuous 
function on $\omega$. Then the function 
\begin{equation}
\label{eq:minPoisson-funct}
   u(x) = \inf \Big\{\int^{2\pi}_0 \phi(f(\E^{\imath t}))\, 
   \frac{dt}{2\pi} \colon \ f\in \Mgot(\d,\omega,x) \Big\}, \quad x\in \omega,
\end{equation}
is minimal plurisubharmonic on $\omega$ or identically $-\infty$; 
moreover, $u$ is the supremum of the minimal plurisubharmonic 
functions on $\omega$ which are not greater than $\phi$. 
\end{theorem}

\begin{proof}
Given a minimal plurisubharmonic function $v\in\MPsh(\omega)$ such
that $v\le \phi$ and a point $x\in \omega$,  the maximum principle for subharmonic functions shows that
\[
	v(x)\le \int^{2\pi}_0 v(f(\E^{\imath t}))\,  \frac{dt}{2\pi} 
	\le \int^{2\pi}_0 \phi(f(\E^{\imath t}))\, \frac{dt}{2\pi}, 
	\quad  \forall f\in\Mgot(\d,\omega,x).
\]
By taking the infimum over all $f$  we obtain $v\le u$ on $\omega$,
where $u$ is defined by (\ref{eq:minPoisson-funct}). To complete the proof we show 
that $u$ is minimal plurisubharmonic.
Let $\Phi$ be the upper semicontinuous function on $\Omega=\omega\times \imath \r^3$
defined by $\Phi(x+\imath y)=\phi(x)$ for $x\in\omega$ and $y\in\r^3$. 
Fix $z=x+\imath y\in\Omega$. Since $\Phi\circ g=\phi\circ \Re g$ for all 
$g\in \Ngot(\d,\Omega,z)$ and any $f\in \Mgot(\d,\omega,x)$ 
is the real part of a null disc $g\in \Ngot(\d,\Omega,z)$, we have
\[ 
	\inf \Big\{\int^{2\pi}_0 \Phi(g(\E^{\imath t}))\, 
   \frac{dt}{2\pi} \colon \ g\in \Ngot(\d,\Omega,z) \Big\} =   
   \inf \Big\{\int^{2\pi}_0 \phi(f(\E^{\imath t}))\, \frac{dt}{2\pi} \colon \ f\in \Mgot(\d,\omega,x) \Big\}.
\]
The left hand side defines a function $U(z)$  which is 
null plurisubharmonic on $\Omega$ by Theorem \ref{th:nullpshminorant}. 
Comparing the two sides we see that $U$ is independent of the imaginary component $y=\Im z$.
Hence the function on the right hand side, which equals $u$ (\ref{eq:minPoisson-funct}), 
is minimal subharmonic on $\omega$ by Lemma \ref{lem:minimal}.
\end{proof}

%
%
%
\begin{definition}\label{def:minimalhull}
The {\em minimal hull}  of  a compact set $K \subset \r^n$ $(n\ge 3)$ is the set 
\begin{equation}\label{eq:minimalhull}
	\wh K_\Mgot=\{x\in \r^n\colon u(x)\le \sup_K u \quad \forall u\in \MPsh(\r^n)\}.
\end{equation}
\end{definition}

We clearly have 
\begin{equation}\label{eq:comhulls}
	K\subset   \wh K_{\Mgot}\subset   \Co(K),
\end{equation}
where $\Co(K)$ denotes the convex hull of $K$.

The maximum principle for subharmonic functions implies that for any bounded minimal surface 
$M\subset \r^n$ with boundary $bM\subset K$ we have $M\subset \wh K_{\Mgot}$.

%
%
%
%
\begin{remark}\label{rem:Mhull}
Since minimal plurisubharmonic functions can be approximated  by
smooth minimal plurisubharmonic functions, we see that our definition 
of the minimal hull coincides with the {\em $2$-convex hull} 
of Harvey and Lawson \cite[Definition 3.1, p.\ 157]{Harvey-Lawson2013}. 
In \cite[Sec.\ 4]{Harvey-Lawson2013} they introduced {\em minimal current hull}
of dimension $p$ of a compact set in any Riemannian manifold $(X^n,g)$ for $2\le p<n$
and showed that it is contained in the $p$-plurisubharmonic hull  of $K$
\cite[Theorem 4.11]{Harvey-Lawson2013}. 
\qed\end{remark}

We have the following analogue of Proposition \ref{prop:exhaustion} whose
proof we leave to the reader; it closely follows the standard plurisubharmonic case
(cf.\ \cite[Theorem 1.3.8, p.\ 25]{Stout-convexity}).

%
%
%
%
\begin{proposition}\label{prop:Mexhaustion}
Given a compact minimally convex set $K=\wh K_\Mgot$ in $\r^n$ and an open set $U\subset \r^n$
containing $K$, there exists a smooth minimal plurisubharmonic exhaustion function 
$\rho\colon\r^n\to\r_+$ such that $\rho=0$ on a neighborhood of $K$ and 
$\rho$ is positive minimal strongly plurisubharmonic on $\r^n\setminus U$.
\end{proposition}

Theorem \ref{th:minpshminorant} implies the following characterization of the minimal hull 
of a compact set in $\r^3$ by sequences of conformal minimal discs. 
This shows that our definition of the minimal hull is a natural one.
Another evidence to this effect is furnished by Theorem \ref{th:minimalcurrents} below
which identifies the minimal hull of a compact set in $\r^3$ with its minimal current hull.

%
%
%
%
\begin{corollary}\label{cor:minullhull}
Let $K$ be a compact set in $\r^3$, and let  $\omega\Subset \r^3$ be a bounded open convex set 
containing $K$. A point $p \in \omega$ belongs to the minimal hull $\widehat K_\Mgot$ of $K$ 
if and only if there exists a sequence of conformal minimal discs $f_j\in \Mgot(\d,\omega,p)$ such that 
\begin{equation}  \label{eq:Mdiscs}
	\big| \{t\in [0,2\pi]: \dist(f_j(e^{\imath t}),K)<1/j\}  \big| \ge 2\pi -1/j,
	\qquad j=1,2,\ldots.
\end{equation}
\end{corollary}

\begin{proof}
Note that Proposition \ref{komp_psh_disc}-(iv) also holds for minimal plurisubharmonic functions,
with essentially the same proof: 
given a compact set $K$ in $\r^3$, an open convex set $\omega\subset \r^3$ 
containing $K$, and a function $u\in\MPsh(\omega)$, 
there exists a function $v\in\MPsh(\r^3)$ such that $v=u$ on  $K$, 
$v$ is strictly convex on $\r^3\setminus\omega$, and $v(x)>\max_K v$ for each $x\in\r^3\setminus\omega$.
This implies that, similarly to Lemma \ref{lem:Nhull-Runge}, we have 
\[
	\wh K_{\Mgot} =\{x\in \omega: v(x)\le \max_K v\ \ \ \forall v\in \MPsh(\omega)\}.
\]
Now Corollary \ref{cor:minullhull} follows from Theorem \ref{th:minpshminorant} by analogous 
observations as Corollary \ref{cor:DD} follows from Theorem \ref{th:nullpshminorant};
we leave out the obvious details.
\end{proof}

We end this section by explaining the relationship between the minimal hull and the null hull 
(see Definition \ref{def:nullhull}). 
Let $\pi\colon \c^n\to\r^n$ be the projection $\pi(x+\imath y)=x$.
By Lemma \ref{lem:minimal} a minimal plurisubharmonic function 
$u$ on $\r^n$ lifts to a null plurisubharmonic function $u\circ\pi$ on $\c^n$.
This implies that for any compact set $L\subset\c^n$ we have
\begin{equation}\label{eq:nullminullhull}
 	\pi \bigl(\wh L_{\Ngot}\bigr) \subset  \wh {\pi(L)}_\Mgot.
\end{equation}
The inclusion may be strict: take $L\subset\c^3$ to be a smoothly embedded Jordan curve 
such that $K=\pi(L)\subset \r^3$ is also a smooth Jordan curve. 
Then $K$ bounds a minimal surface $M$ which is therefore contained in 
$\wh {K}_\Mgot$. However, if for some point $p\in L$ the tangent line 
$T_p L$ does not belong to the null quadric $\Agot$ (\ref{eq:Agot}), 
then by Corollary \ref{cor:NHcurves} we have $\wh L_{\Ngot}=L$. 

For a more precise result in this direction see Corollary \ref{cor:currentnullhull} below.

%
%
%
%
%
\section{Green currents}
\label{sec:Gcurrents}

In this section we obtain some technical results that will be used in the sequel.
For the general theory of  currents we refer to Federer \cite{Federer} or Simon \cite{Simon}; 
see also Morgan \cite{Morgan} for a reader friendly introduction to the subject.

Let $\zeta =x+\imath y$ be the coordinate on $\c \cong\r^2$.
The {\em Green current}, $G$, on the closed unit disc $\cd$
is defined on any $2$-form $\alpha = adx\wedge dy$ 
with $a\in\Cscr(\cd)$ by
\begin{equation}\label{eq:Greencurrent}
	G(\alpha) = -\frac{1}{2\pi}  \int_\d \log|\zeta|\cdotp \alpha =
	- \frac{1}{2\pi}  \int_\d \log|\zeta|\cdotp a(\zeta) dx\wedge dy.
\end{equation}
Clearly $G$ is a positive current of bidimension $(1,1)$. 
For any function $u$ of class $\Cscr^2$ on a domain in $\c$ 
we have $dd^c u = 2\imath \di\bar\di u = \triangle u\cdotp dx\wedge dy$. 
Green's formula
\begin{equation}\label{eq:Green}
	u(0) = \frac{1}{2\pi} \int_0^{2\pi} u(e^{\imath t}) dt  + 
	\frac{1}{2\pi} \int_\d \log|\zeta| \cdotp \triangle u(\zeta) \, dx\wedge dy,
\end{equation}
which holds for any $u\in \Cscr^2(\cd)$, tells us that 
\[
	(dd^c G) (u) =G(dd^c u) = \frac{1}{2\pi} \int_0^{2\pi} u(e^{\imath t}) dt  - u(0).
\]
This means that $dd^c G = \sigma-\delta_0$, where $\sigma$ is the normalized 
Lebesgue measure on the circle $b\d=\t$ and $\delta_x$ is the evaluation at a point  $x$.
If  $u$ is subharmonic then 
\[
	0\le G(\triangle u\cdotp dx\wedge dy) = dd^c G (u) = \int_{\t} u\, d\sigma - u(0).
\]

Let $x=(x_1,\ldots, x_n)$ be the coordinates on $\r^n$.
Given a smooth map $f=(f_1,\ldots,f_n)\colon\cd \to\r^n$ we denote by 
$f_*G$ the 2-dimensional current on $\r^n$ given on any 2-form 
$\alpha=\sum_{i,j=1}^n a_{i,j} dx_i\wedge dx_j$  by 
\begin{equation}\label{eq:imageofG}
	(f_*G)(\alpha)=G(f^*\alpha) =
	-\frac{1}{2\pi}  \int_\d \log|\zeta| \cdotp f^*\alpha.
\end{equation}
Clearly $\supp (f_*G) \subset f(\cd)$. We call $f_*G$ the {\em Green current supported by $f$}.

Assume now that $f\colon\cd\to\c^n$ is a $\Cscr^2$ map that is 
holomorphic on $\d$. (Identifying $\c^n$ with $\r^{2n}$, the Cauchy-Riemann
equations imply that $f$ is conformal harmonic, except at the critical points
where $df=0$.) Since $f$ commutes with the $\dibar$-operator, and 
hence also with the conjugate differential 
$d^c=\imath(\dibar-\di)$, we get for any $u\in\Cscr^2(\c^n)$ that 
\[
	dd^c (f_*G) (u) = f_*G (dd^c u) = G(f^* dd^c u) =
	G(dd^c (u\circ f)) = \int_\t (u\circ f) d\sigma - (u\circ f)(0). 
\]
This gives the following well known formula:
\begin{equation}\label{eq:ddc}
	dd^c (f_*G) = f_* \sigma - \delta_{f(0)};
\end{equation}
see e.g.\ Duval and Sibony \cite[Example 4.9]{Duval-Sibony1995}. 

Recall the following representation theorem (see Federer \cite[\S 4.1.5--\S 4.1.7]{Federer}).
Associated to any $p$-dimensional current $T$ on $\r^n$ 
of finite mass  there is a positive Radon measure $||T||$ on $\r^n$ 
and a $||T||$-measurable frame of unit $p$-vectors $\vec{T}$ such that for any $p$-form $\alpha$ 
with compact support on $\r^n$ we have that
\begin{equation}\label{eq:Talpha}
	T(\alpha) = \int_{\r^n} \langle \alpha, \vec{T}\rangle \, d\,||T||.
\end{equation}
Here $\langle \alpha, \vec{T}\rangle$ denotes the value of $\alpha$
on the $p$-vector $\vec{T}$ (a $||T||$-measurable function on $\r^n$).
The mass $\M(T)$ of the current $T$ is then given by
\begin{equation}\label{eq:Tmass}
	\M(T) = \sup \{T(\alpha)\colon |\langle \alpha,\vec{T}\rangle|\le 1\}  = \int_{\r^n}  d\,||T||.
\end{equation}
In particular, every current of finite mass is representable by integration.

Wold proved  that for any {\em holomorphic} disc 
$f\colon \cd\to\c^n$ the mass $\M(f_*G)$ is bounded in terms of the dimension $n$ 
and of $\sup_{\zeta\in \cd} |f(\zeta)|$ (cf.\ \cite[Lemma 2.2]{Wold2011}). 
The following lemma  gives an explicit formula for a bigger  class  of Green currents.
What is important for our purposes is that $\M(f_*G)$ is bounded by the $L^2$-norm 
of $f$ on the circle.

%
%
%
%
\begin{lemma}\label{lem:mass}
If $f=(f_1,\ldots,f_n)\colon\cd \to\r^n$  is a conformal harmonic immersion 
of class $\Cscr^2(\cd)$, then the mass of the Green current $f_*G$ satisfies
\begin{equation}\label{eq:massT}
	\M(f_*G) \le \frac{1}{4} \left( \int_\t |f|^2 d\sigma - |f(0)|^2 \right).
\end{equation}
If $f$ is injective outside of a closed set of measure zero in $\cd$, or if $f\colon \cd\to\c^n$ 
is a holomorphic disc, then we have equality in (\ref{eq:massT}).
\end{lemma}

\begin{proof} 
Denote the partial derivatives of $f:\cd\to\r^n$ by $f_x$ and $f_y$.  Write 
\[
	|f|^2=\sum_{i=1}^n f_i^2, \qquad 
	|\nabla f|^2=\sum_{i=1}^n |\nabla f_i|^2 = 
	\sum_{i=1}^n \left(f_{i,x}^2 + f_{i,y}^2 \right).
\]
We identify $f_x$ and $f_y$ with the vector fields $f_*\frac{\partial}{\partial x}$ and 
$f_*\frac{\partial}{\partial y}$, respectively.  
Since $f$ is conformal,  these vector fields  are orthogonal and satisfy $|f_x|=|f_y|$. Let
\begin{equation} \label{eq:vecT}
	\vec{T}=\frac{f_x\wedge f_y}{|f_x|\cdotp |f_y|} =\frac{f_x\wedge f_y}{|f_x|^2}.
\end{equation}
Given a $2$-form $\alpha$ on $\r^n$, we have
\begin{equation} \label{eq:f*alpha}
	f^*\alpha = \langle\alpha\circ f, f_x\wedge f_y\rangle \, dx \wedge dy =
	 \langle \alpha\circ f, \vec{T}\rangle\, |f_x|^2 \, dx\wedge dy.
\end{equation}
The definition of $T=f_* G$ (\ref{eq:imageofG})  and the formula (\ref{eq:f*alpha}) imply
\begin{equation} \label{eq:Talpha2}
	 T(\alpha) = -\frac{1}{2\pi}  \int_\d \log|\zeta| \cdotp 
	 \langle \alpha\circ f, \vec{T}\rangle \cdotp  |f_x|^2 \, dx\wedge dy.
\end{equation}
From the definition of the mass of a current and (\ref{eq:Talpha2}) it follows that
\begin{equation}\label{eq:mass}
	\M(T) = \sup \{T(\alpha)\colon |\langle \alpha,\vec{T}\rangle|\le 1\}
	\le -\frac{1}{2\pi}  \int_\d \log|\zeta| \cdotp  |f_x|^2 \, dx\wedge dy.
\end{equation}

So far we have only used the hypothesis that $f$ is conformal.
At this point we take into account that $f$ is also harmonic. 
For any harmonic function $v \in \Cscr^2(\cd)$ we have
\[ 
	dd^c v^2 = d(2v d^c v) = 2 dv\wedge d^c v = 2|\nabla v|^2 dx\wedge dy.
\]
Applying this to each component $f_i$ of the harmonic map $f=(f_1,\ldots,f_n)$ we get 
\[
	|\nabla f|^2 \, dx\wedge dy = \sum_{i=1}^n | \nabla f_i|^2 \, dx\wedge dy 
	= \frac{1}{2}\sum_{i=1}^n dd^c f^2_i.
\]
Inserting the identity
$
	|f_x|^2 \, dx\wedge dy = \frac{1}{2} |\nabla f|^2  \, dx\wedge dy 
	= \frac{1}{4} \sum_{i=1}^n dd^c f^2_i
$
into  (\ref{eq:mass}) and applying Green's identity
(\ref{eq:Green}) gives the inequality  (\ref{eq:massT}) for $\M(f_*G)$.
\end{proof}

%
%
%
%
\begin{remark}\label{rem:loss}  
The formula (\ref{eq:mass}) shows that the mass measure of $T=f_*G$  (\ref{eq:Talpha}) equals  
\begin{equation}\label{eq:massmeasure}
	||T||(U)= -\frac{1}{2\pi}  \int_{f^{-1}(U)} \log|\zeta| \cdotp  |f_x|^2 \, dx\wedge dy
\end{equation}
for any open set $U\subset\r^n$ such that $f$ is injective on $f^{-1}(U)\subset\cd$.
The possible loss of mass, leading to a strict inequality in (\ref{eq:mass}) and hence in (\ref{eq:massT}),
may be caused by the cancellation of parts of the immersed surface $f(\cd)$ 
(considered as a current) due to the reversal of the orientation of the frame field $\vec T$. 
(This can happen for example by immersing the disc conformally onto a M\"obius band in $\r^3$
and letting $\vec T$  be its tangential frame field.)
If $f$ is injective outside a closed set $E\subset \cd$ of measure zero, then for any open set
$V\supset E$ we can easily find a $2$-form $\alpha$ such that $|\langle \alpha,\vec{T}\rangle| \le 1$
on $f(\cd)$ and $\langle \alpha,\vec{T}\rangle = 1$ on $f(\cd\setminus V)$. 
By shrinking $V$ down to $E$ we see that the inequality in (\ref{eq:mass}) 
becomes an equality. A holomorphic disc $f\colon \cd\to\c^n$
carries a canonical orientation induced by the complex structure,
so there is no cancellation of mass in the current $f_*G$.
On the other hand, the total mass of $f(\cd)$, counted with multiplicities and 
with the weight induced by  $-\log|\cdotp|$, 
always equals the expression on the right hand side of (\ref{eq:massT}).
\qed\end{remark}

Given a $\Cscr^2$ function $u$ on a domain in $\r^n$, we denote by 
\[
	\Hess u=\left(\frac{\partial^2 u}{\partial x_i \partial x_j}\right)
\]
its Hessian, a quadratic form on $T_x\r^n$.
If $\vec{T}_x$ is a unit $p$-frame at $x$, we denote by  $\tr_{\vec{T}_x} (\Hess u)$ 
the trace of the Hessian of $u$ restricted to the $p$-plane $\span\,\vec{T}_x\subset T_x\r^n$.

%
%
%
%
\begin{lemma}\label{lem:ddc-uf}
For every conformal harmonic immersion $f\colon \cd\to\r^n$ we have
\begin{equation}\label{eq:ddcuf}
	dd^c(u\circ f) = \left(\tr_{\vec{T}} (\Hess u)\circ f \right) \cdotp  |f_x|^2\, dx\wedge dy,
\end{equation}
where $\vec{T}$ is the unit $2$-frame along $f(\cd)$ given by (\ref{eq:vecT}).
\end{lemma}

\begin{proof}
Note that $d^c(u\circ f) = \sum_{j=1}^n \left(\frac{\di u}{\di x_j}\circ f\right) \cdotp d^c f_j$ 
and hence
\[
	dd^c(u\circ f) = \sum_{i,j=1}^n  d\left( \frac{\di u}{\di x_j}\circ f\right) \wedge d^c f_j
	= \sum_{i,j=1}^n \left(\frac{\di^2 u}{\di x_i \di x_j}\circ f \right) \cdotp df_i\wedge d^c f_j.
\]
(We used that $dd^c f_j=0$ since $f$ is harmonic.)  We also have 
\[
	df_i\wedge d^c f_j = \left(f_{i,x} dx + f_{i,y} dy\right) \wedge  
	\left(-f_{j,y} dx + f_{j,x} dy\right) = 
	\left(f_{i,x}f_{j,x}   +f_{i,y}  f_{j,y} \right) dx\wedge dy.
\]
Inserting this identity into the previous formula yields
\begin{eqnarray*}
	dd^c(u\circ f) &=& \sum_{i,j=1}^n \left(\frac{\di^2 u}{\di x_i \di x_j}\circ f\right)  
	\bigl(f_{i,x}f_{j,x}   +f_{i,y}  f_{j,y} \bigr) dx\wedge dy \cr
	 &=& \left(f_x^t \cdotp  (\Hess u)\circ f  \cdotp f_x +  f_y^t \cdotp  (\Hess u)\circ f  \cdotp f_y  \right) dx\wedge dy \cr
	&=&  \left(\tr_{\vec{T}} (\Hess u)\circ f \right) \cdotp  |f_x|^2\, dx\wedge dy.
\end{eqnarray*}	 
(We used that $|f_x|=|f_y|$ and $f_x \cdotp f_y=0$ since $f$ is conformal.) 
This gives (\ref{eq:ddcuf}). 
\end{proof}

\begin{remark}\label{rem:Laplace} 
The formula (\ref{eq:ddcuf}) actually means that 
\[
	\tr_{\vec{T}} (\Hess u)= \triangle_M (u|_M), 
\]
where $\triangle_M$ is the Laplace-Beltrami operator on the minimal 
surface $M= f(\cd)\subset \r^n$ in the induced metric.
(See e.g.\ \cite[Proposition 2.10]{Harvey-Lawson2009} or 
(2.10) in \cite{Harvey-Lawson2013}.  If $M$ is not minimal then
there is an error term; see (2.9) in \cite{Harvey-Lawson2013}.)
\qed\end{remark}

%
%
%
%
\begin{lemma}\label{lem:Hessian}
Let $T=f_*G$ be the Green current (\ref{eq:imageofG}) on $\r^n$ supported by a conformal
harmonic immersion $f:\cd\to\r^n$, and let $\vec{T}$ be a unit tangential $2$-frame 
along $f(\cd)$. For every $\Cscr^2$ function $u$ on a neighborhood of $f(\cd)$
we have that
\begin{equation}\label{eq:Hess}
	\int_{\r^n} \tr_{\vec{T}} (\Hess u) \, d\, ||T|| \le  
	- \frac{1}{2\pi}  \int_\d \log|\zeta|\cdotp dd^c(u\circ f)
	=\int_0^{2\pi} \! \! u\bigl(f(e^{\imath t})\bigr) \frac{dt}{2\pi} - u(f(0)).	
\end{equation}
The equality holds under the same conditions as in Lemma \ref{lem:mass}.
\end{lemma}

\begin{proof}
Let $\vec T$ be as in (\ref{eq:vecT}). From the expression (\ref{eq:massmeasure}) 
for  $||T||$ we get that
\[
	\int_{\r^n} \tr_{\vec{T}} (\Hess u) \cdotp d\, ||T|| \le
		 -\frac{1}{2\pi}  
	\int_\d \log |\zeta| \cdotp\left( \tr_{\vec{T}} (\Hess u)\circ f \right) \cdotp  |f_x|^2\, dx\wedge dy.
\]
The equality holds under the same conditions as described in Lemma \ref{lem:mass}.
Now (\ref{eq:Hess}) follows immediately  from the identity (\ref{eq:ddcuf}).
\end{proof}

Lemma \ref{lem:Hessian} gives another proof of the mass inequality (\ref{eq:mass})  for $\M(f_*G)$: 
apply  (\ref{eq:Hess}) with the function $u(x)=\sum_{j=1}^n x_j^2$ and observe that
$\tr_{\vec{T}} (\Hess u)\equiv 4$ for any $\vec T$.

%
%
%
%
%
\section{Characterizations of  minimal hulls and null hulls by Green currents}
\label{sec:currents}

Recall that a $2$-form $\alpha$ on a domain $\Omega\subset \c^n$ is said to be {\em positive}
if for every point $p\in \Omega$ and vector $\nu\in T_p\c^n$ 
we have $\langle \alpha(p), \nu\wedge J\nu\rangle \ge 0$. (Here
$J$ denotes the standard complex structure operator on $\c^n$.) 
Let $\alpha=\alpha^{2,0}+\alpha^{1,1}+\alpha^{0,2}$ be  a decomposition 
according to type. Since $\langle \alpha(p), \nu\wedge J\nu\rangle$ vanishes for forms
of type $(2,0)$ and $(0,2)$, we see that $\alpha$ is positive if and only if $\alpha^{1,1}$ is. 
A current $T$ of bidimension $(1,1)$ on $\c^n$ is positive if $T(\alpha)\ge0$ 
for every positive $(1,1)$ form $\alpha$ with compact support.

%
%
%
%
\begin{definition} 
A $2$-form $\alpha$ on a domain $\Omega\subset \c^n$ is  {\em null positive}
if for every point $p\in \Omega$ and null vector $\nu\in \Agot^*$ (\ref{eq:Agot})
we have $\langle \alpha(p), \nu\wedge J\nu\rangle \ge 0$. (We identify
$\nu$ with a tangent vector in $T_p\c^n$, and $J$ denotes the complex structure operator.) 
A $(1,1)$-current $T$ on $\c^n$ is null positive if $T(\alpha)\ge 0$ for every null positive 
$2$-form $\alpha$ with compact support.
\end{definition} 

Note that a function $u$ of class $\Cscr^2$ is null plurisubharmonic (see Definition \ref{def:npsh})
if and only if the (1,1)-form $dd^c u$ is null positive.

If $f\colon \cd\to\c^n$ is a null holomorphic disc, then $T=f_* G$ is a null positive current.
Indeed, if $\alpha$ is a null positive $2$-form and $\vec{T}$ is a positively oriented orthonormal 
frame field along $f$, then $\langle \alpha, \vec{T}\rangle\ge 0$ and
hence $T(\alpha)=\int \langle \alpha, \vec{T}\rangle \, d\,||T||\ge 0$ by (\ref{eq:Talpha}).

The following result is analogous to the characterization of the polynomial hull,
due to  Duval and Sibony \cite{Duval-Sibony1995} and Wold \cite[Theorem 2.3]{Wold2011}
(see Theorem \ref{th:DS} above).

%
%
%
%
\begin{theorem} 
\label{th:nullcurrents}
Let $K$ be a compact set in $\c^3$. A point $p\in\c^3$ belongs to the null hull 
$\wh K_\Ngot$ of $K$ (\ref{eq:nullhull}) if and only if 
there exists a null positive $(1,1)$-current $T$ on $\c^3$ with compact support  satisfying 
$dd^c T=\mu-\delta_p$, where $\mu$ is a probability measure  on $K$ and $\delta_p$
is the point mass at $p$, such that
\begin{equation}\label{eq:submeanvalue} 
	u(p) \le \int_K \! u \, d\mu  \qquad \forall  u\in \NPsh(\c^3).
\end{equation}
The support of any such current $T$ is contained in the null hull $\wh K_\Ngot$ of $K$.
\end{theorem}

\begin{proof} 
If $\mu$ is a probability measure satisfying  (\ref{eq:submeanvalue}),  
then $u(p) \le \int_K \! u d\mu \le\max_K u$ for every $u\in \NPsh(\c^3)$, and hence 
$p\in \wh K_\Ngot$. (This implication holds on $\c^n$ for any $n\ge 3$.)

For the converse implication we follow  Wold \cite{Wold2011}.
Choose a ball $\Omega\subset \c^3$ containing $K$. 
Let $f_j\in \Ngot(\d,\Omega,p)$ be a sequence of null discs furnished by Corollary \ref{cor:DD}.
The Green currents $T_j=(f_j)_*(G)$ have uniformly bounded masses
by Lemma \ref{lem:mass}. Consider the $T_j$'s as continuous linear functionals
on the separable Banach space of differential forms on $\c^3$
with the sup-norm topology. A bounded set of functionals is metrizable \cite[V.5.1]{Conway}, 
and hence the weak* compactness of a closed bounded set of currents
coincides with the sequential weak* compactness. 
Hence a subsequence of $T_j$ converges weakly
to a null positive $(1,1)$-current $T$ with compact support
and finite mass. By (\ref{eq:ddc}) we have $dd^c T_j =(f_j)_*\sigma -\delta_p$ for all 
$j$. Condition (\ref{eq:Ndiscs}) implies that the supports of the probability 
measures $\sigma_j=(f_j)_*\sigma$ converge to $K$; passing to
a subsequence we obtain a probability measure $\mu=\lim_{j\to\infty} \sigma_j$
on $K$. It follows that
\[
	dd^c T=\lim_{j\to\infty}  dd^c T_j =
	 \lim_{j\to\infty}  \sigma_j  -\delta_p= \mu-\delta_p.
\]
If $u\in\NPsh(\c^3)$ then  $dd^c u$ is null positive, and hence
$0\le T(dd^c u)=\int_K u\, d\mu - u(p)$. 

It remains to show that $\supp\, T\subset  \wh K_\Ngot$.  This is a special case of 
part (a) in the following proposition. Part (b) will be used in  
Theorem \ref{th:minimalcurrents} below. We consider $\r^n$ as the standard real
subspace of $\c^n$ and denote by $\pi\colon\c^n\to\r^n$ the projection
$\pi(x+\imath y)=x$.

%
%
%
%
\begin{proposition}\label{prop:support}
Let $T$ be a null positive current of bidimension $(1,1)$ on $\c^n$. 
\begin{itemize}
	\item[\rm (a)] If $T$ has compact support and satisfies $dd^c T\le 0$ on $\c^n\setminus K$
	for some compact set $K\subset\c^n$, then $\supp\, T \subset  \wh K_\Ngot$.
	\item[\rm (b)] Assume that $T$ has bounded mass and $\pi(\supp\, T)\subset \r^n$ 
	is a bounded subset of $\r^n$. 
	 If $dd^c T\le 0$ on $\c^n\setminus (K\times \imath\r^n)$ 
	for some compact set $K\subset\r^n$, then $\supp\, T \subset  \wh K_\Mgot \times \imath \r^n$.
	(Recall that $\wh K_\Mgot $ is the minimal hull of $K$.)
\end{itemize}	
\end{proposition}

\begin{proof}[Proof of (a)]  
Fix a point $q\in \c^n\setminus \wh K_\Ngot$.
Proposition \ref{prop:exhaustion} furnishes a nonnegative smooth  function $u\in \NPsh(\c^n)$ 
which is strongly null plurisubharmonic on a neighborhood $U\subset\c^n$ of $q$ and vanishes on a 
neighborhood of $\wh K_\Ngot$. Since the support of $u$ is contained in 
$\c^n\setminus K$ where $dd^c T\le 0$, we have $T(dd^c u) = (dd^c T)(u) \le 0$. 
(We are using that $T$ has compact support, so it may be applied to forms with arbitrary supports.)
As $T$ is null positive on $\c^n$, we also have $T(dd^c u)\ge 0$;
hence $T(dd^c u) = 0$. Since $u$ is strongly null plurisubharmonic on $U$, 
it follows that $\M(T|_U)=0$. This proves that $\supp\, T\subset  \wh K_\Ngot$.

\noindent {\em Proof of (b).}  
Write $\Tcal_U=\pi^{-1}(U) = U\times\imath \r^n$ for any $U\subset \r^n$. 
Choose a ball $B\subset \r^n$ such that $\pi(\supp\, T) \cup \wh K_\Mgot \subset B$.
Since $T$ has bounded mass, it can be applied to any $2$-form 
with bounded continuous coefficients on  the tube $\Tcal_B$.
In particular, for any function $u\in \Cscr^2(\r^n)$ the current $T$ can be applied
to the $2$-form $dd^c (u\circ\pi)$.
Fix a point $q\notin \wh K_\Mgot$. Proposition \ref{prop:Mexhaustion}
furnishes a smooth nonnegative function $u\in \MPsh(\r^n)$ that 
is minimal strongly  plurisubharmonic on a neighborhood $U\subset\r^n$ of $q$
and vanishes  on a neighborhood $V\subset \r^n$ of $\wh K_\Mgot$.  
The function $\tilde u=u\circ \pi$ on $\c^n$ is then null plurisubharmonic
(see Lemma \ref{lem:minimal}), it is strongly null plurisubharmonic on the tube $\Tcal_U$, 
and it vanishes on $\Tcal_V$. Since the support of $\tilde u$ is contained in 
$\c^n\setminus \Tcal_K$ where $dd^c T$ is negative, we have 
$T(dd^c \tilde u)= (dd^c T)(\tilde u) \le 0$.  As $T$ is null positive, 
we also have $T(dd^c \tilde u)\ge 0$; hence $T(dd^c \tilde u) = 0$. 
Since $\tilde u$ is strongly null plurisubharmonic on $\Tcal_U$, it follows that 
$T$ has no mass there. This proves that $\supp\, T\subset  \wh K_\Mgot \times\imath\r^n$.
\end{proof}
\vskip -3mm This completes the proof of Theorem \ref{th:nullcurrents}. \end{proof}

In the remainder of the section we obtain several characterizations of the minimal hull 
of a compact set in $\r^3$. Recall that $\pi:\c^3\to \r^3$ is the projection $\pi(x+\imath y)=x$.

%
%
%
%
\begin{theorem}[\bf Characterization of the minimal hull by currents] 
\label{th:minimalcurrents}
Let $K$ be a compact set in $\r^3$. A point $p\in \r^3$ belongs to the minimal hull $\wh K_\Mgot$  
(\ref{eq:minimalhull}) if and only if there exists a null positive current $T$ on $\c^3$ 
of finite mass such that $\pi(\supp\, T)\subset\r^3$ is a bounded set
and $dd^c T = \mu-\delta_p$, where $\mu$ is a probability measure on 
the tube $\Tcal_K = K\times \imath\r^3$. 
\end{theorem}

Any current $T$ satisfying Theorem \ref{th:minimalcurrents}
has support contained in $\wh K_\Mgot \times \imath \r^3$ according to
Proposition \ref{prop:support}-(b).
If $T$ and $\mu$ are as in Theorem \ref{th:minimalcurrents} then 
\begin{equation}\label{eq:submeanvalue3}
	u(p) \le \int_{\Tcal_K} (u\circ \pi) \, d\mu \le \max_K u \qquad \forall u\in \MPsh(\r^3).
\end{equation}
Indeed, if $u\in \MPsh(\r^3)$ then $\tilde u = u\circ \pi\in \NPsh(\c^3)$ 
by Lemma \ref{lem:minimal}, and $\tilde u$ is bounded
on $\Tcal_B=B\times \imath\r^3$ for every bounded set $B\subset \r^3$. Choosing 
$B$ to be a large ball, we have $\supp\, T \subset \Tcal_B$ and hence 
$0\le T(dd^c \tilde u) = \int \tilde u\, d\mu - u(p)$, thus proving (\ref{eq:submeanvalue3}).
(Here we use that $T$ has bounded mass, so it can be applied to any $2$-form 
with bounded continuous coefficients on  $\Tcal_B$. Note that $dd^c \tilde u$ is such
when $\tilde u = u\circ \pi$ and $u$ is a $\Cscr^2$ function on $\r^3$.)
The projection $\nu=\pi_* \mu$ is then a probability measure on $K$
satisfying $u(p)\le \int_{K} u \, d\nu$ for every  $u\in \MPsh(\r^3)$.
Hence Theorem \ref{th:minimalcurrents} implies the following corollary.

%
%
%
%
\begin{corollary} 
\label{cor:minimalJensen}
Let $K$ be a compact set in $\r^3$. A point $p\in \r^3$ belongs to the minimal hull $\wh K_\Mgot$ 
(\ref{eq:minimalhull}) if and only if there exists a probability measure $\nu$ on $K$ such that
\begin{equation}\label{eq:submeanvalue1} 
	u(p) \le \int_K \! u \, d\nu  \le \max_K u \qquad \forall  u\in \MPsh(\r^3).
\end{equation}
\end{corollary}

A measure $\nu$ satisfying (\ref{eq:submeanvalue1}) is called a {\em minimal Jensen measure} for the point  $p\in \wh K_\Mgot$.

%
%
\begin{proof} [Proof of Theorem \ref{th:minimalcurrents}]
If $\mu$ is a probability measure on $\Tcal_K$ satisfying (\ref{eq:submeanvalue3}),
then the measure $\nu =\pi_* \mu$ on $K$ satisfies (\ref{eq:submeanvalue1})
and hence  $p\in \wh K_\Mgot$.

Let us now prove the converse. Fix a point $p\in \wh K_\Mgot$.
Corollary \ref{cor:minullhull} furnishes  a bounded sequence of conformal minimal 
discs $f_j \in \Mgot(\d,\r^3,p)$ satisfying (\ref{eq:Mdiscs}). 
We may assume that each $f_j$ is real analytic on a neighborhood of $\cd$.
Let $g_j$ be the harmonic conjugate of $f_j$ on $\cd$ with $g_j(p)=0$.
Then $F_j=f_j+\imath g_j \in \Ngot(\d,\c^3)$ is a null holomorphic disc.
Let $\Theta_j = (f_j)_*G$ and $T_j=(F_j)_* G$ be the associated Green currents
on $\r^3$ and $\c^3$, respectively. Then $T_j$ is null positive and $\pi_* T_j = \Theta_j$
for every $j$.  By Lemma \ref{lem:mass} we have
\[
	4 \M(T_j) = \int_\t |F_j|^2 d\sigma - |p|^2 
	= \int_\t |f_j|^2\,d\sigma + \int_\t |g_j|^2\,d\sigma - |p|^2.
\]
Since the conjugate function operator is bounded on the Hilbert space $L^2(\t)$
\cite[Theorem 3.1, p.\ 116]{Garnett} and the sequence $f_j$ is uniformly bounded, 
we see that $\M(T_j)\le C<\infty$ for some constant $C$ and for all $j\in \n$.
We may assume by passing to a  subsequence that $T_j$ converges  weakly to a null positive 
$(1,1)$-current $T$ with finite mass (but not necessarily with compact support 
since the harmonic conjugates $g_j$ of $f_j$ need not be uniformly bounded), 
and $\Theta_j$ converges to a $2$-dimensional current $\Theta$ on $\r^3$.
From $\pi_*  T_j=\Theta_j$ for all $j\in \n$ we also get that $\pi_* T=\Theta$. 

By (\ref{eq:ddc}) we have that $dd^c T_j =(F_j)_*\sigma -\delta_p$ for all 
$j\in \n$. Note that $\pi_* (F_j)_*\sigma = (f_j)_* \sigma$. 
Condition (\ref{eq:Mdiscs}) implies that the supports of the probability 
measures $(f_j)_*\sigma$ converge to $K$, and hence the supports
of the measures $(F_j)_*\sigma$ converges to the tube $\Tcal_K$.
By passing to a subsequence we obtain a probability measure 
$\mu=\lim_{j\to\infty} (F_j)_*\sigma$ on $\Tcal_K$. It follows that
\[
	dd^c T=\lim_{j\to\infty}  dd^c T_j =
	 \lim_{j\to\infty}  (F_j)_*\sigma  -\delta_p= \mu-\delta_p.
\]

The last claim in Theorem \ref{th:minimalcurrents} follows directly from Proposition \ref{prop:support}-(b).
\end{proof}

%
%
%
%
\begin{remark}\label{rem:HL}
The proof of Proposition \ref{prop:support} is similar to the proof of the following result  
due to Harvey and Lawson \cite[Theorem 4.11]{Harvey-Lawson2013}:
{\em Given a compact set $K \subset \r^n$ and a minimal $p$-dimensional current $T$
on $\r^n$ \cite[Definition 4.7]{Harvey-Lawson2013} with $\supp(\partial T)\subset K$, 
it follows that the support of $T$ is contained in the $p$-plurisubharmonic hull of $K$.}
It is not clear from their work whether every point in the $p$-plurisubharmonic hull of $K$
lies in the support of such current. The main new part of Theorem \ref{th:minimalcurrents} for
$n=3$ and $p=2$ is  that it completely explains the minimal hull by Green currents. 
It would be of interest to know whether the analogous
result holds in higher dimensional Euclidean spaces.
\qed\end{remark}

We wish to compare the minimal hull of a compact set $K\subset \r^n$ to the null hull of the tube
$\Tcal_K= K\times \imath \r^n \subset \c^n$. The latter set is unbounded, and the standard
definition of its polynomial hull (and, by analogy, of its null hull) is by exhaustion with compact sets. 
Let $B_r\subset \r^n$ denote the closed ball of radius $r$ centered at the origin. 
Then $\Tcal_K= \bigcup_{r>0}  \Tcal_{K,r}$ where $\Tcal_{K,r} =  K\times \imath \overline B_r$, and we set
\begin{equation}\label{eq:hull}
	\widehat {\Tcal_K} = \bigcup_{r>0}  \widehat{\Tcal_{K,r}},\qquad 
	\widehat {(\Tcal_K)}_\Ngot  = \bigcup_{r>0}  \widehat{ (\Tcal_{K,r})}_\Ngot.
\end{equation}
Clearly $\widehat {(\Tcal_K)}_\Ngot \subset  \widehat {\Tcal_K}$.
From (\ref{eq:nullminullhull}) we also get that that 
\[
	\widehat {(\Tcal_K)}_\Ngot \subset  \wh {K}_\Mgot \times \imath\r^n \subset \Co (K)\times \imath\r^n.
\]
We do not know whether the first inclusion could be strict. On the other hand, 
Theorem \ref{th:minimalcurrents} motivates the following definition of the current null hull
of the tube $\Tcal_K$.

%
%
%
%
\begin{definition}\label{def:currentnullhull}
Let $K$ be a compact set in $\r^3$ and $\Tcal_K=K\times\imath\r^3\subset\c^3$ be the tube
over $K$. The  {\em current null hull} of $\Tcal_K$, denoted $\wh{(\Tcal_K)}_{\Ngot^*}$, is 
the union of supports of all null positive $(1,1)$-currents $T$  on $\c^3$ with finite mass
such that $\pi(\supp\, T)\subset \r^3$ is a bounded set and $dd^c T\le 0$ on $\c^3\setminus \Tcal_K$.
\end{definition}

Theorem  \ref{th:nullcurrents} shows that $\wh{(\Tcal_K)}_\Ngot \subset \wh{(\Tcal_K)}_{\Ngot^*}$.
Now Theorem \ref{th:minimalcurrents} implies the following result which extends the classical 
relationship between conformal minimal discs and null holomorphic discs to the corresponding hulls.

%
%
%
%
\begin{corollary}\label{cor:currentnullhull}
If $K$ is a compact set in $\r^3$ and $\Tcal_K=K\times\imath\r^3\subset\c^3$, then 
\begin{equation}\label{eq:projection} 
	\wh{(\Tcal_K)}_{\Ngot^*} = \wh K_\Mgot \times \imath\r^3.
\end{equation}
Here $\wh{(\Tcal_K)}_{\Ngot^*}$ denotes the current null hull of $\Tcal_K$  (see Definition \ref{def:currentnullhull}).
\end{corollary}

\medskip
\begin{question}
Let $T$ be a current as in Theorem \ref{th:minimalcurrents} with $dd^c T=\mu-\delta_p$,
where $\mu$ is a probability measure on $\Tcal_K$ and $p\in\r^3$.
Is the point $p$ always contained in the support of the projected current $\Theta=\pi_* T$
on $\r^3$?
\end{question}

If so, we could conclude that $\wh K_\Mgot$ is the union
of supports of currents of the form $\Theta=\lim_{j\to\infty} (f_j)_*G$, where
$f_j$ is a bounded sequence of conformal minimal discs as in Corollary \ref{cor:minullhull}
whose boundaries converge to $K$ in the measure theoretic sense.
However, the problem  is that cancellation of mass may occur in $\Theta$; 
see Remark \ref{rem:loss}. This can be circumvented by considering
$(f_j)_*G$ as bounded  linear functionals  of the space of quadratic forms
on $\r^3$. The clue is given by Lemma \ref{lem:Hessian}; we now explain this.

Denote by $\Qcal(\r^n)$ the separable Banach space consisting of all quadratic forms 
$
	h=\sum_{i,j=1}^n h_{i,j}(x) dx_i\otimes dx_j
$
on $\r^n$ with continuous coefficients $h_{i,j}$  and with 
finite sup-norm 
\[
	||h||=\sum_{i,j=1}^n \sup_{x\in\r^n} |h_{i,j}(x)| <\infty.
\]
Assume that $T$ is 2-dimensional current on $\r^n$ of finite mass
(hence representable by integration), and let $\vec T$ and $||T||$ be its
frame field and mass measure (\ref{eq:Talpha}), respectively. Then $T$ defines a
bounded linear functional on $\Qcal(\r^n)$ by the formula
\[
	T(h) = \int_{\r^n} \tr_{\vec{T}} h \cdotp d\, ||T||,\quad h\in \Qcal(\r^n),
\]
where $\tr_{\vec{T}} h$ is the trace of the restriction of $h$ to the $2$-plane $\span\, \vec T$.
Since $\tr_{\vec{T}} h$  is  independent of the orientation determined by $\vec T$,
every compact surface $M\subset \r^n$ (also nonorientable)
defines a bounded linear functional on $\Qcal(\r^n)$.
More generally, one can use rectifiable surfaces with finite total area, that is,  countable unions of images 
of Lipshitz maps $f:\cd\to \r^n$. (Such surfaces define rectifiable currents, see \cite{Federer}.) 

Given a $\Cscr^1$ immersion $f:\cd\to \r^n$, we denote by $\vec T$ 
the $2$-frame along $f$ determined by the partial derivatives $f_x,f_y$. Define a  bounded linear functional
$T=f_*G$ on $\Qcal(\r^n)$ by
\begin{equation}\label{eq:Th}
	T(h) =  - \frac{1}{2\pi}  \int_\d  \log|\zeta| \left( \tr_{\vec{T}} h \circ f \right)\cdotp dx\wedge dy,
	\quad h\in \Qcal(\r^n).
\end{equation}
If $f$ is conformal harmonic, we see from (\ref{eq:Hess}) 
that for every $u\in \Cscr^2(\r^n)$ we have
\begin{equation} \label{eq:Hess2}
	T(\Hess u) = - \frac{1}{2\pi}  \int_\d \log|\zeta|\cdotp dd^c(u\circ f) 
	=\int_0^{2\pi} \! \! u\bigl(f(e^{\imath t})\bigr) \frac{dt}{2\pi} - u(f(0)).	
\end{equation}
(There is no cancellation of mass as explained above, so we have the equality in  (\ref{eq:Hess}).)
This gives the following characterization of the minimal hull in $\r^3$.

%
%
%
%
\begin{corollary} \label{cor:hessian}
Let $K$ be a compact set in $\r^3$. A point $p\in \r^3$ belongs to the minimal hull $\wh K_\Mgot$ 
of $K$ (\ref{eq:minimalhull}) if and only if there exist a continuous linear functional $T$ with compact 
support on $\Qcal(\r^3)$ and a probability measure $\mu$ on $K$ such that 
\begin{equation}\label{eq:THess}
	T(\Hess u) = \int_K u\, d\mu - u(p) \qquad \forall u\in \Cscr^2(\r^3),
\end{equation}
and $T(\Hess u)\ge 0$ for every minimal plurisubharmonic function $u$ of class $\Cscr^2$ on $\r^3$.
The support of every such functional $T$ is contained in $\wh K_\Mgot$.
\end{corollary}

Note that the measure $\mu$ in (\ref{eq:THess}) is a minimal Jensen measure for $p$
(see Corollary \ref{cor:minimalJensen}).

\begin{proof}
If $T$ and $\mu$ as in the corollary exist, 
then for every $u\in \Cscr^2(\r^3)\cap \MPsh(\r^3)$ we have
$0\le T(\Hess u) = \int_K u d\mu - u(p)$, and hence $p\in \wh K_\Mgot$.

Assume now that $p\in \wh K_\Mgot$. 
Let $f_j\colon \cd\to\r^3$ be a bounded sequence of conformal harmonic immersions
as in Corollary \ref{cor:minullhull}, with $f_j(0)=p$ for all $j$. The  associated
linear functionals $T_j$ on $\Qcal(\r^3)$, given by (\ref{eq:Th}),
are a bounded sequence in the dual space $\Qcal(\r^3)^*$. 
By the same argument as in the proof of Theorem \ref{th:nullcurrents}
we see that a subsequence of $T_j$ converges in the weak* topology
to a bounded linear functional $T\in \Qcal(\r^3)^*$. 
Similarly, we may assume that  the probability measures 
$(f_j)_*\sigma$ converge weakly to a measure $\mu$ on $K$. Since every $T_j$ satisfies
(\ref{eq:Hess2}), we get in the limit the identity (\ref{eq:THess}). If $u\in\Cscr^2$ is minimal 
plurisubharmonic then $T_j(\Hess u)\ge 0$ by (\ref{eq:Hess2}) and the submeanvalue
property of the subharmonic function $u\circ f_j$ on $\cd$. Passing to the limit we obtain
$T(\Hess u)\ge 0$. Proposition \ref{prop:support}-(b) shows that $\supp\, T\subset \wh K_\Mgot$.
\end{proof}


\section{Bochner's tube theorem for polynomial hulls}
\label{sec:Bochner}

Bochner's tube theorem \cite{Bochner} says that for every connected 
open set $\omega\subset\r^n$ the envelope of holomorphy of the tube $\Tcal_\omega=\omega\times\imath\r^n$ 
equals its convex hull $\Co(\Tcal_\omega) = \Co(\omega)\times \imath\r^n$.
This beautiful classical result can  be found in most standard texts 
on complex analysis (see e.g.\ \cite{Bochner-Martin,Hormander-SCV,Stout-convexity}),
and it was extended in several directions by different authors.
A new recent proof was given by J.\ Hounie \cite{Hounie} where the 
reader can find updated references.

In light of Bochner's theorem the following is a natural question:

\begin{question}
Let $K$ be a connected compact set in $\r^n$.
Is the polynomially convex hull  $\wh \Tcal_K$ (\ref{eq:hull}) of the tube
$\Tcal_K=K \times\imath\r^n\subset\c^n$ always equal its convex hull:
\[
 	\wh \Tcal_K \stackrel{?}{=} \Co(\Tcal_K) = \Co(K) \times\imath\r^n.
\]	
(The inclusion $\wh \Tcal_K\subset \Co(\Tcal_K)$ is obvious.)
\end{question}

We give an affirmative answer if  the polynomial hull 
of $\Tcal_K$ is defined as the union of supports
of positive currents $T$ of bidimension $(1,1)$ with finite mass
such that $dd^c T$ is negative on $\c^n\setminus \Tcal_K$ and
$\supp \,T$ projects to a bounded subset of $\r^n$. 
This definition is entirely natural in view of the 
Duval-Sibony-Wold \cite{Duval-Sibony1995,Wold2011}  characterization 
of the polynomial hull of a compact set $L\subset\c^n$
by positive $(1,1)$ currents $T$ with compact support 
such that $dd^c T\le 0$ on $\c^n\setminus L$
(see Theorem \ref{th:DS} in Sect.\ \ref{sec:intro}).

Let $\pi\colon \c^n\to\r^n$ denote the projection $\pi(x+\imath y)=x$.
We have the following result. 

%
%
%
%
\begin{theorem} [\bf Bochner's tube theorem for polynomial hulls] \label{th:Bochner}
Let $K$ be a connected compact set in $\r^n$. For every point 
$z_0=p+\imath q\in \c^n$ with $p\in\Co(K)$ 
there exists a positive current $T$ of bidimension $(1,1)$  on $\c^n$ with finite mass
satisfying $\supp(T)\subset \Co(K)\times\imath\r^n$ and
$dd^c T=\mu-\delta_{z_0}$, where  $\mu$ is a probability measure on $\Tcal_K$
and $\delta_{z_0}$ is the Dirac mass at $z_0$.

Conversely, let $T$ be a positive  $(1,1)$  current on $\c^n$
with finite mass such that $\pi(\supp\, T)$ is a bounded set in $\r^n$.
If $dd^c T\le 0$ on $\c^n\setminus \Tcal_K$, then 
$\supp\, T \subset  \Co(K) \times \imath \r^n$.
\end{theorem}

The analogous result for null hulls is given by Corollary \ref{cor:currentnullhull}.

\begin{proof}[Proof of Theorem \ref{th:Bochner}]
By translation invariance in the $\imath\r^n$ direction it suffices to prove the result
for points $z_0=p\in \r^n$. We shall use the following simplest case of the 
{\em convex integration lemma} due to Gromov (cf.\ \cite{Gromov:convex} or \cite{Gromov:book}).

\begin{lemma}\label{lem:convex}  
Let $\omega$ be a connected open set in $\r^n$ and $p\in\r^n$ a point in the
convex hull $\Co(\omega)$. Then there exists a smooth loop
$g \colon \t=b\d\to \omega$ such that $\int_0^{2\pi} g(e^{\imath t}) \frac{dt}{2\pi}=p$.
\end{lemma}

The proof of Lemma \ref{lem:convex} is quite simple: write $p=\sum_{j=1}^k c_j p_j$ where 
$p_j\in \omega$ and $c_j>0$ with $\sum_{j=1}^k c_j=1$.
Pick a smaller connected bounded open set $\omega'\Subset \omega$ which contains the 
points $p_1,\ldots, p_k$. Choose a smooth path $g\colon\t\to\omega'$ such that it spends almost 
the time $c_j$ at the point $p_j$ and goes very quickly from one point to the next 
in the mean time. This gives a loop in $\omega'$ whose integral is as close as desired to $p$; 
by a small translation of $g$ we can  ensure that it lies in $\omega$ and that its integral
equals $p$. 

Since every smooth map $\t\to\r^n$ is the boundary value of a harmonic map
$g:\cd\to \r^n$ and we have $\int_0^{2\pi} g(e^{\imath t}) \frac{dt}{2\pi}=g(0)$, 
we immediately get the following corollary.

\begin{corollary}\label{cor:convex}  
Let $\omega$ and $p$ be as in Lemma \ref{lem:convex}. Then 
there exists an analytic disc $f\colon \cd\to\c^n$ such that 
$f(0)=p$ and $f(\t)\subset \omega\times\imath\r^n$.
\end{corollary}

Assume now that $K$ is a compact connected set  in $\r^n$ and $p\in \Co(K)$.
For each $j\in\n$ let $\omega_k=\{x\in \r^n\colon \dist(x,K)<1/j\}$.
Corollary \ref{cor:convex} furnishes a holomorphic disc $f_j\colon \cd\to\c^n$ 
with $f_j(0)=p$ and $f_j(\t)\subset \omega_j \times \imath\r^n$.
The sequence of real parts $g_j=\Re f_j \colon \cd\to\omega_j$ is uniformly bounded, and
hence bounded in $L^2(\t)$. Since the harmonic conjugate transform is a bounded operator on $L^2(\t)$ 
(see \cite[Theorem 3.1, p.\ 116]{Garnett}), we have $\int_\t |f_j|^2 d\sigma \le C < \infty$
for some constant $C$ and all $j\in\n$. Let $T_j=(f_j)_* G$,  a positive $(1,1)$ current on $\c^n$.
Lemma \ref{lem:mass} implies that $\M(T_j)\le C/4<\infty$ for all $j$.
We now proceed as in Theorem \ref{th:minimalcurrents}.
Passing  to a  subsequence we may assume that currents $T_j$ converge weakly to a positive 
$(1,1)$-current $T$ on $\c^n$ with finite mass whose support lies over a bounded subset of $\r^n$, 
and the measures $\sigma_j=(f_j)_*\sigma$ converge weakly 
to a probability measure $\mu$ supported on $K\times\imath\r^n$.
(No mass is lost when passing to the limit since the sequence $f_j$ is  
bounded in $L^2(\t)$.) By (\ref{eq:ddc}) we have that $dd^c T_j =(f_j)_*\sigma -\delta_p$ for all 
$j\in \n$, and hence we get $dd^c T=\mu-\delta_p$.

The last claim in Theorem \ref{th:Bochner} is proved as in Proposition \ref{prop:support}-(b),
using the  fact for any convex function $u$ on $\r^n$ the function $u\circ\pi$ is plurisubharmonic
on $\c^n$.
\end{proof}

\bigskip
\centerline{ADDED IN PROOF}

In the recent preprint "Minimal surfaces in minimally convex domains" by A.\ Alarc\'on, B. Drinovec Drnov\v sek, F.\ Forstneri\v c 
and F.\ J.\ L\'opez, (http://arxiv.org/abs/1510.04006), it is shown that the main results of the present paper hold in any 
dimension $n\ge 3$. 


\bigskip

\centerline{ACKNOWLEDGEMENTS}

B.\ Drinovec Drnov\v sek and F.\ Forstneri\v c are supported in part by the research program P1-0291 
and the grant J1-5432 from ARRS, Republic of Slovenia. 

The authors wish to thank Antonio Alarc\'{o}n and Francisco J.\ L\'opez for 
helpful discussions regarding minimal surfaces and the notion of the minimal hull, 
Roman Drnov\v sek for his advice concerning sequential compactness in weak* topologies,
and Reese Harvey and Blaine Lawson for their remarks on an earlier version of the paper.
We also thank the referee for having pointed out the connection with universal domains
for minimal surfaces and with mean-convex domains (cf.\ Remark \ref{rem:mean-convex} above).


\vskip 5mm

\noindent Barbara Drinovec Drnov\v sek

\noindent Faculty of Mathematics and Physics, University of Ljubljana, and Institute
of Mathematics, Physics and Mechanics, Jadranska 19, SI--1000 Ljubljana, Slovenia.

\noindent e-mail: {\tt barbara.drinovec@fmf.uni-lj.si}

\vskip 5mm
\noindent Franc Forstneri\v c

\noindent Faculty of Mathematics and Physics, University of Ljubljana, and Institute
of Mathematics, Physics and Mechanics, Jadranska 19, SI--1000 Ljubljana, Slovenia.

\noindent e-mail: {\tt franc.forstneric@fmf.uni-lj.si}

\end{document}